\documentclass[11pt]{amsart}

\usepackage{amssymb,amsmath,accents}
\usepackage{bbm,pgf,tikz}
\usetikzlibrary{arrows,automata}
\usepackage{graphicx}
\usepackage{a4wide}

\usepackage[colorlinks=true, pdfstartview=FitV, linkcolor=blue, 
citecolor=blue]{hyperref}

\newtheorem{theorem}{Theorem}[section]
\newtheorem{prop}[theorem]{Proposition}
\newtheorem{lemma}[theorem]{Lemma}
\newtheorem{fact}[theorem]{Fact}
\newtheorem{coro}[theorem]{Corollary}
\theoremstyle{definition}
\newtheorem{definition}[theorem]{Definition}

\newtheorem{remark}[theorem]{Remark}

\newcommand{\ts}{\hspace{0.5pt}}
\newcommand{\nts}{\hspace{-0.5pt}}

\newcommand{\RR}{\mathbb{R}\ts}
\newcommand{\CC}{\mathbb{C}}
\newcommand{\NN}{\mathbb{N}}

\newcommand{\cA}{\mathcal{A}}
\newcommand{\cB}{\mathcal{B}}
\newcommand{\cC}{\mathcal{C}}
\newcommand{\cD}{\mathcal{D}}
\newcommand{\cE}{\mathcal{E}}

\newcommand{\cM}{\mathcal{M}}
\newcommand{\cP}{\mathcal{P}}
\newcommand{\pa}{\hphantom{g}\nts\nts}

\newcommand{\tiM}{\ts\widetilde{\nts\nts M}\nts\nts}
\newcommand{\whM}{\ts\widehat{\nts\nts M}\nts\nts}
\newcommand{\tir}{\tilde{r}}
\newcommand{\whr}{\hat{r}}
\newcommand{\fa}{\mathfrak{a}}
\newcommand{\fb}{\mathfrak{b}}
\newcommand{\fc}{\mathfrak{c}}

\newcommand{\ee}{\ts\mathrm{e}}
\newcommand{\one}{\mathbbm{1}}
\newcommand{\dm}{\ts\mathsf{d}}
\newcommand{\bell}{\ts\mathsf{B}}
\newcommand{\bd}{{\, \scriptstyle \boxdot \,}}
\newcommand{\trans}{{\scriptscriptstyle \mathsf{T}}}
\newcommand{\pmin}{\ts\ts\underline{\nts\nts 0\nts\nts}\ts\ts}
\newcommand{\pmax}{\ts\ts\underline{\nts\nts 1\nts\nts}\ts\ts}
\newcommand{\klein}{\overset{\raisebox{-2pt}{.}}{\prec}}
\newcommand{\al}{{\ts{\curvearrowleft}\ts\ts}}
\newcommand{\sm}{\ts{\setminus}\ts}

\newcommand{\udo}[1]{\underaccent{$\text{.}$}{#1\ts}\nts}
\newcommand{\exend}{\hfill$\Diamond$}

\newcommand{\card}{\mathrm{card}}
\newcommand{\Mat}{\mathrm{Mat}}

\newcommand{\defeq}{\mathrel{\mathop:}=}
\newcommand{\eqdef}{=\mathrel{\mathop:}}

\newcommand{\myfrac}[2]{\frac{\raisebox{-2pt}{$#1$}}
  {\raisebox{0.5pt}{$#2$}}}

\begin{document}

\title[Recombination and Markov embedding]
 {Recombination in discrete and continuous time\\[2mm]
  from the viewpoint of Markov embedding}

\author{Ellen Baake}
\address{Technische Fakult\"at, Universit\"at Bielefeld, 
         Postfach 100131, 33501 Bielefeld, Germany}
\email{ebaake@techfak.uni-bielefeld.de}

\author{Michael Baake}
\address{Fakult\"at f\"ur Mathematik, Universit\"at Bielefeld, 
         Postfach 100131, 33501 Bielefeld, Germany}
\email{mbaake@math.uni-bielefeld.de}
         
\author{Jeremy Sumner}         
\address{School of Natural Sciences, Discipline of Mathematics,
         University of Tasmania,
    \newline \indent Private Bag 37, Hobart, TAS 7001, Australia}
\email{Jeremy.Sumner@utas.edu.au}

\begin{abstract} 
  The classic recombination equation, both in discrete and in
  continuous time, can be solved in a way that derives from the Markov
  chain of a partitioning process. Here, we revisit this structure
  from the point of view of the Markov embedding problem.  In
  particular, we analyse when a discrete-time Markov matrix of
  recombination type can occur in a time-homogeneous Markov semigroup
  that is generated by a recombination rate matrix.  En route, we also
  show that such rate matrices (or Markov generators) generally do not
  form a matrix algebra, but span a real Lie algebra.
\end{abstract}


\maketitle

\section{Introduction}

Recombination is an important genetic mechanism that mixes (or
reshuffles) the genetic material of different individuals from
generation to generation; it takes place in the course of sexual
reproduction, and is an important mechanism of
evolution. Traditionally, in the limit of large population sizes, it
was studied via the \emph{recombination equation}, which is a
well-known deterministic dynamical system from mathematical population
genetics \cite{Chris,Buerger,fast}. Some substantial progress was made
in recent years by a change of perspective on the process via
reversing time and looking back into the past. Then, recombination
implies that the genes of an individual are partitioned across its
parents, grandparents, grand-grandparents, and so on. The resulting
partitioning process is a Markov chain in discrete or continuous
time. Both play important roles in population genetics and its
applications \cite{Chris,Buerger,BS2012}, and they are also relevant
to phylogenetics \cite{GM96,SH,EMB}; we refer to \cite[Sec.~5.4]{HSW},
\cite[Secs.~3.3 and 8.4]{Durett} as well as \cite[Sec.~7.2.4]{Wakeley}
for background and general overviews.

The process is best known in the context of its graphical
representation, the \emph{ancestral recombination graph}, which goes
back to \cite{Hudson83,GM96}. Here, it contains both fragmentation and
coalescence events, which make it highly complex; it has been studied
intensely in recent years; compare \cite{Mano13,JFS15,EPB16,
  LM-PS21,Alberti24}, to name just a few. In the
law-of-large-numbers-regime, the process turns into one of pure
fragmentation, which is more accessible. It has been studied in both
discrete and continuous time; we refer to \cite{BvW14,fast,Martinez},
as well as to \cite{EMB} for a review with further references. A
question that has remained open so far is how the discrete-time and
the continuous-time versions are related. More precisely: Under which
conditions on the parameters can the discrete-time Markov chain be
embedded into the semigroup of an underlying continuous-time and
time-homogeneous Markov chain? This is a biologically relevant
instance of the classic Markov embedding problem
\cite{Elfving,King}. The goal of the paper is to fill this gap and to
study the relation between discrete-time and continuous-time
recombination in some detail.  \smallskip

The paper is organised as follows. In Section~\ref{sec:prelim}, we
collect some material and results on Markov embedding and on
partitions of finite sets, together with some additional background
results that we need. We then analyse the partitioning processes of
recombination in Section~\ref{sec:reco-part}, with some emphasis on
the underlying algebraic structure. This section also derives some
general results on the embedding problem for this class of processes.
We then tackle recombination for two and three sites explicitly in
Section~\ref{sec:2-3}, where the embedding problem is solved
completely, and in explicit form with a clear-cut
interpretation. These two cases are still fairly simple and do not
display the general structure, as we show for four sites in
Section~\ref{sec:4}. Here, we see the non-linear parameter dependence
kick in, as well as the absence of a matrix algebra
structure. Instead, we derive that the recombination generators span a
$14$-dimensional Lie algebra, and fully solve the embedding problem in
this case. The algebraic structure is then generalised to an arbitrary
number of sites in Section~\ref{sec:general}, where we also derive a
systematic hierarchy of invariant subspaces and a corresponding tensor
product structure. Further, we derive criteria that are more concrete
than those from Section~\ref{sec:reco-part}.

\section{Notions, preliminaries and preparatory results}\label{sec:prelim}

Let us introduce the notation for the various mathematical objects we
will use, and give some general results that we shall need later.

\subsection{Markov embedding}

Let us briefly introduce and summarise the Markov embedding problem.
A \emph{Markov matrix}
$M = (M^{}_{ij})^{}_{1\leqslant i,j \leqslant \dm}$ has non-negative
entries such that all row sums are $1$. A \emph{Markov generator}
$Q = (Q^{}_{ij} )^{}_{1\leqslant i,j \leqslant \dm}$ has non-negative
off-diagonal entries and zero row sums. Such a $Q$ is also known as a
\emph{rate matrix}; compare \cite{Norris}.  Rate matrices give rise to
time-homogeneous, Abelian \emph{Markov semigroups} of the form
$\{ \ee^{t \ts Q} : t \geqslant 0 \}$, all elements of which are
Markov matrices. Such semigroups always contain $\one$ and thus are
monoids.

A Markov matrix $M$ is called \emph{embeddable} when it satisfies
$M=\ee^Q$ for some rate matrix $Q$; see \cite{Elfving,King} for the
origins and \cite{BS1,BS2} for further background material and recent
results. If embeddable, $M$ occurs within a (time-homogeneous) Markov
semigroup.  If $Q$ is a Markov generator, then so is $\alpha \ts Q$
for any $\alpha >0$, and the chosen time scale is immaterial.  All
embeddable Markov matrices are infinitely divisible, because $M=\ee^Q$
has $\ee^{Q/n}$ as an $n$-th Markov root, for all $n\in\NN$.  The
embedding problem for $\dm=2$ was solved by Kendall, as stated in
\cite{King}; see \cite{BS3} for the general situation with
$\dm\leqslant 4$.  Let us recall Kendall's result.

\begin{fact}[Kendall]\label{fact:King}
  A Markov matrix\/ $M=\left(\begin{smallmatrix} 1-a & a \\ b & 1-b
      \end{smallmatrix}\right)$, where\/ $a,b \in [0,1]$,
    is embeddable if and only if\/ $0 < \det (M) = 1-a-b \leqslant 1$.
    Equivalently, this is true if and only if\/ $0\leqslant a+b < 1$.
    In this case, the embedding\/ $M=\ee^Q$ is unique, with the rate
    matrix\/ $\ts Q= - \frac{\log (1-a-b)}{a+b} (M \nts -\one)$, which
    includes the case\/ $Q=0$ for\/ $a=b=0$.  \qed
\end{fact}

\begin{figure}
\centerline{\includegraphics[width=0.6\textwidth]{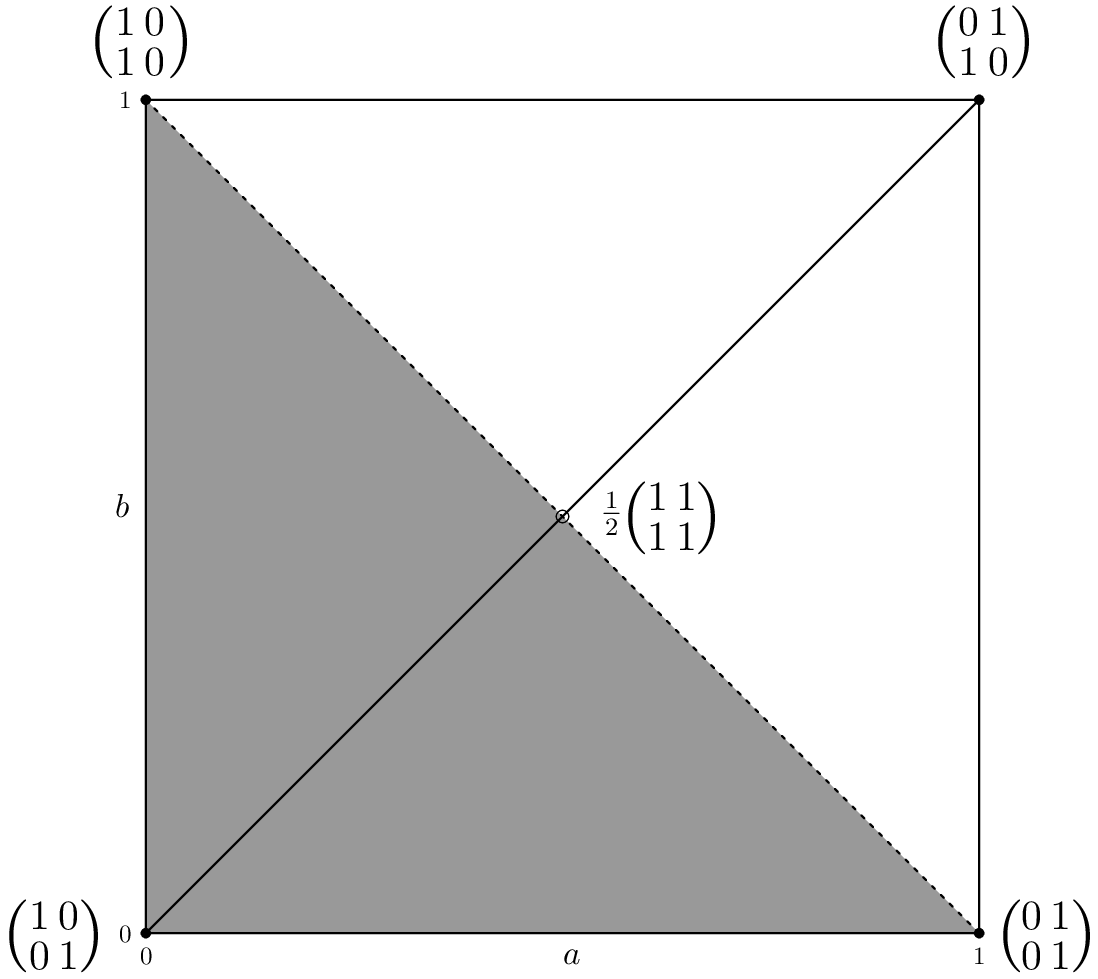}}
\caption{Graphical representation of Kendall's general embedding
  result for $2\ts{\times}\ts 2$ Markov matrices; see text for
  details. \label{fig:2d} }
\end{figure}

The situation is transparently summarised in Figure~\ref{fig:2d},
where the convex set $\cM^{}_{2}$ of all Markov matrices for $\dm=2$
forms the (closed) square, with the extremal ones as corners. The grey
triangle, excluding the dashed line, comprises the embeddable
ones. The matrices on the excluded line are the non-trivial
idempotents in $\cM_2$, which are infinitely divisible but singular;
see \cite{BS2} for more on the role of idempotents.

More generally, since $\det (\ee^Q) = \ee^{\mathrm{tr} (Q)}$, an
embedding of $M$ is only possible if $\det (M) > 0$, which is tied to
the existence of a real matrix logarithm of $M$. The latter property,
and conditions for uniqueness, was investigated by Culver
\cite{Culver}, and can be summarised via the (complex) \emph{Jordan
  normal form} (JNF) of $M$ as follows, which simplifies to
diagonalisable matrices in an obvious way.
  
\begin{fact}[Culver]\label{fact:Culver}
  A matrix\/ $B\in\Mat (d,\RR)$ has a real logarithm if and only if
  the following two conditions are satisfied.
\begin{enumerate}\itemsep=2pt
\item[(C1)] The matrix\/ $B$ is non-singular.
\item[(C2)] Each elementary Jordan block of the JNF of\/ $B$ that
  belongs to an eigenvalue on the negative real axis occurs with even
  multiplicity.
\end{enumerate}   
Further, the real logarithm of\/ $B$ is unique if and only if all
eigenvalues of\/ $B$ are positive real numbers and no elementary
Jordan block of\/ $B$ occurs more than once.  \qed
\end{fact}

Let us also recall that non-uniqueness of the real logarithm emerges
either from the existence of a complex-conjugate pair of eigenvalues
(giving a countably infinite set of solutions) or from the presence of
more than one copy of an elementary Jordan block (resulting in an
uncountable set of solutions); see \cite{Culver} for details. However,
uniqueness can be restored by some additional constraints, one case of
which will become important to us later. Employing
\cite[Sec.~12.4]{LT}, we explain this in some detail, as it is less
standard in the literature.

Let $J^{}_{k} = \lambda \one^{}_{k} + N^{}_{k}$ with $k\geqslant 2$ be
an elementary (upper-triangular) Jordan block with $\lambda > 0$,
where $N^{}_{k}$ is the nilpotent matrix with $1$s on the first
superdiagonal and $0\ts$s everywhere else. It is nilpotent of index
(or degree) $k$, so $N^k_k = 0$ and $N^m_k \ne 0$ for
$0\leqslant m < k$.  By Fact~\ref{fact:Culver}, $J^{}_{k}$ has a
unique real logarithm, which is given by the principal logarithm
\begin{equation}\label{eq:J-log}
  L^{}_{k} \, = \, \log (J^{}_{k}) \, = \, \log (\lambda) \one^{}_{k} + 
  \sum_{m\geqslant 1} \myfrac{(-1)^{m-1}}{m \ts \lambda^m} 
  N^{m}_{k} \, = \, \log (\lambda) \one^{}_{k} + 
  \sum_{m = 1}^{k-1} \myfrac{(-1)^{m-1}}{m \ts \lambda^m} 
  N^{m}_{k}  ,
\end{equation}
where the last step follows because $N^m_k = 0$ for all
$m\geqslant k$. Let us now consider two elementary Jordan blocks with
the same $\lambda > 0$, say $J^{}_{k}$ and $J^{}_{\ell}$, with
possibly different $k,\ell\in\NN$, and the block-diagonal matrix
$J^{}_{k} \oplus J^{}_{\ell}$. When $k\ne \ell$, its real logarithm is
unique, but we include this case for reasons that will become clear
shortly.

Now, we look for upper-triangular block matrices
$R=\left( \begin{smallmatrix} \alpha & \gamma \\ 0 &
    \beta \end{smallmatrix} \right)$ such that
$\ee^R = J^{}_{k} \oplus J^{}_{\ell}$.  Due to the block structure,
this forces $\ee^{\alpha} = J^{}_{k}$ and $\ee^{\beta} = J^{}_{\ell}$
and thus $\alpha = L^{}_{k}$ and $\beta = L^{}_{\ell}$ from
Eq.~\eqref{eq:J-log}.  Now, we have $[ \ee^R, R ]=0$, which gives one
non-trivial constraint, namely $J^{}_{k} \gamma = \gamma J^{}_{\ell}$
or, equivalently, $N^{}_{k} \gamma = \gamma N^{}_{\ell}$. Since the
same commutation relation holds for powers of the nilpotent matrices,
we must also have $ L^{}_{k} \gamma = \gamma L^{}_{\ell}$. Calculating
the exponential of $R$ gives
\begin{equation}\label{eq:J-sum}
    \ee^R \, = \, \exp \begin{pmatrix} L^{}_{k} & \gamma \\
      0 & L^{}_{\ell} \end{pmatrix} \, = \, \begin{pmatrix}
      J^{}_{k} & C \\ 0 & J^{}_{\ell} \end{pmatrix} ,
\end{equation}
where the (generally rectangular) matrix $C$ reads
\[
    C \, = \sum_{n=1}^{\infty} \myfrac{1}{n \ts !} \sum_{m=0}^{n-1}
    L^{m}_{k} \gamma L^{n-1-m}_{\ell} \, =  \sum_{n=1}^{\infty}
    \myfrac{n}{n\ts !} \ts L^{n-1}_{k} \gamma \, = \, \exp (L^{}_{k}) 
    \ts \gamma \ts .
\]
Note that the penultimate step uses the commutation relation derived
previously. Now, due to $\ee^R = J^{}_{k} \oplus J^{}_{\ell}$, we must
have $C=0$ in \eqref{eq:J-sum}, which implies
$\gamma = \exp (-L^{}_{k}) 0 = 0$, and we see that the restriction to
upper-triangular block matrices confirms uniqueness of the real
logarithm for $k\ne \ell$ and restores it for $k=\ell$. This has the
following important consequence.

\begin{lemma}\label{lem:J-block}
  Let\/ $J^{}_{k_1}, \ldots , J^{}_{k_r}$ be\/ $r$ elementary Jordan
  blocks, all for the same\/ $\lambda >0$. Then, the upper-triangular
  block matrix\/ $B=J^{}_{k_1} \!\oplus \cdots \oplus J^{}_{k_r}$ has
  precisely one real logarithm with upper-triangular block form,
  namely\/ $R = L^{}_{k_1} \!\oplus \cdots \oplus L^{}_{k_r}$.
\end{lemma}

\begin{proof}
  If $B = \ee^R$ with $R$ upper triangular, $R$ can be structured in
  blocks that fit the sizes given by the $J^{}_{k_i}$ in $B$, and the
  diagonal blocks of $R$ must be the $L^{}_{k_i}$ from
  \eqref{eq:J-log}. Next, consider any neighbouring pair
  $L^{}_{k_i}\oplus L^{}_{k_{i+1}}$. Using our above argument shows
  that the block just above it must be $0$, thus implying that all
  blocks on the first superdiagonal must vanish.  This allows us to
  employ our argument to next-to-nearest neighbour pairs, and thus to
  conclude that all blocks in the second superdiagonal must vanish as
  well. Inductively, we then get that all blocks above the diagonal
  ones must be $0$, and the claimed uniqueness is established.
\end{proof}

To continue, let $\cM_{\dm}$ denote the set of all Markov matrices in
dimension $\dm$.  We observe that the upper-triangular Markov matrices
form a closed, convex subset of $\cM_{\dm}$, both also being closed
under matrix multiplication.  Likewise, the upper-triangular Markov
generators lie within the algebra of all upper-triangular matrices
with zero row sums. At this point, we state one general result on the
existence of real matrix logarithms as follows.

\begin{lemma}\label{lem:real-log-exists}
  {\nts}Let\/ $M \ts {\in} \ts\ts \cM_{\dm}$ be a non-singular,
  upper-triangular Markov matrix with simple spectrum. Then, all
  eigenvalues of\/ $M$ lie in the half-open interval\/ $(0,1]$.
   
  Moreover, in this case, $M$ has a unique real logarithm,
  $L = \log (M)$. This\/ $L$ has zero row sums and is the principal
  matrix logarithm of $\ts M$, as given by the convergent series
\[
  L \, = \, \log (\one + A) \, = \sum_{m=1}^{\infty}
  \myfrac{(-1)^{m-1}}{m} A^m ,
\]   
where\/ $A = M \nts -\one$ has spectral radius\/ $\varrho^{}_{A} <
1$. In particular, $L$ is upper triangular as well.
\end{lemma}   

\begin{proof}
  Since $M$ is upper triangular, its eigenvalues are the diagonal
  elements.  They all lie in $(0,1]$ because $M$ is Markov and
  non-singular, so $0 < M_{ii} \leqslant 1$ for all
  $1\leqslant i \leqslant d$. Simple spectrum means distinct
  eigenvalues, and Culver's theorem (Fact~\ref{fact:Culver}) implies
  that $M$ possesses a unique real logarithm; see
  \cite[Sec.~2.3]{Higham} or \cite{BS2} for details.

  All eigenvalues of $A$ lie in $(-1,0]$, which implies
  $\varrho^{}_{A} < 1$ and thus the convergence of the series, which
  gives the principal matrix logarithm \cite{HJ,Higham}.  The last
  claim follows because $A$ and all its powers $A^m$ with $m\in\NN$
  are upper triangular with zero row sums, a property which is
  preserved in the limit, as follows from a standard continuity
  argument.  Consequently, $A$ is an element of the matrix algebra
  mentioned above.
\end{proof}

Let us pause to show a better way to calculate $L$, under the
conditions of Lemma~\ref{lem:real-log-exists}. Since $M$ is Markov,
with $1$ being an eigenvalue, $A=M\nts - \one$ is a rate matrix (but
\emph{not} the Markov generator we are after). As such, it lies in the
matrix algebra
\begin{equation}\label{eq:def-A0}
   \cA^{(\dm)}_{\, 0} \, \defeq \, \{ B \in \Mat (\dm ,\RR) :
   \text{ all row sums of $B$ are $0$} \} \ts ,
\end{equation}
which is non-unital, because it neither contains $\one$ nor any other
two-sided unit. Clearly, $\cA^{(\dm)}_{\, 0}$ contains the algebra of
upper-triangular matrices with zero row sums mentioned above.  We have
$L \in \cA^{(\dm)}_{\, 0}$, though $L$ need not be a Markov generator
(we shall shortly see examples where this happens). Still, $M=\ee^L$
means that we can use the \emph{spectral mapping theorem} (SMT). If
$\sigma (B)$ denotes the spectrum of a matrix $B$, the SMT states that
the spectra satisfy the relation $\sigma (\ee^L) = \ee^{\sigma (L)}$,
including multiplicities; see \cite[Thm.~9.4.6]{LT}.  If
$\lambda^{}_{0}, \lambda^{}_{1}, \ldots , \lambda^{}_{\dm-1}$ are the
(positive) eigenvalues of $M$, with $\lambda^{}_{0} = 1$ say,
$\mu^{}_{i} = \lambda^{}_{i} - 1$ are those of $A$, then with
$\mu^{}_{0} = 0$ and $\mu^{}_{i} < 0$ for all
$1\leqslant i \leqslant \dm -1$ because $L$ has simple spectrum. The
matrix $A$ then has a characteristic polynomial of degree $\dm$, say
$P(z)$, with $z$ occurring as one linear factor. The Cayley--Hamilton
theorem then implies that our matrix $L$ from
Lemma~\ref{lem:real-log-exists} satisfies
\[
    L \in \langle A, A^2, \ldots , A^{\dm -1} \rangle^{}_{\RR} \ts ,
\]
which denotes the (non-unital) subalgebra of $\cA^{(\dm)}_{\, 0}$
generated by $A$.

We thus know that $L$ must satisfy
\begin{equation}\label{eq:real-log}
    L \, = \sum_{i=1}^{\dm -1} \alpha^{}_{i} \ts A^{i} ,
\end{equation}
with real coefficients $\alpha^{}_{i}$. The SMT then leads to the
$\dm -1$ equations
\[
    \log (\lambda^{}_{j} ) \, = \sum_{i=1}^{\dm -1}
    \alpha^{}_{i} \, \mu^{\ts i}_{j}
\]
with $1\leqslant j \leqslant \dm -1$ (as $\log (\lambda^{}_{0})=0$
is automatic) and $\mu^{}_{j} = \lambda^{}_{j} -1$. In matrix form,
this reads
\[
    \begin{pmatrix} \log (\lambda^{}_{1}) \\ \vdots \\
    \log (\lambda^{}_{\dm -1}) \end{pmatrix} \, = \, \begin{pmatrix} 
    \mu^{}_{1} & \mu^{2}_{1} & \cdots & \mu^{\dm -1}_{1} \\
    \mu^{}_{2} & \mu^{2}_{2} & \ldots & \mu^{\dm -1}_{2} \\
    \vdots & \vdots & \ddots & \vdots \\ 
      \mu^{}_{\dm -1} & \mu^{2}_{\dm -1}
      & \ldots & \mu^{\dm -1}_{\dm -1} \end{pmatrix}
      \begin{pmatrix} \alpha^{}_{1} \\ \vdots \\ \alpha^{}_{\dm -1}
      \end{pmatrix} \, = \, V  \begin{pmatrix} 
      \alpha^{}_{1} \\ \vdots \\ \alpha^{}_{\dm -1} \end{pmatrix} ,
\]
where $V$ is a simple variant of the Vandermonde matrix, 
compare \cite[Sec.~0.9.11]{HJ}, with
\[
    \det (V) \, = \, \prod_{i=1}^{\dm -1} \mu^{}_{i} \,
    \prod_{1\leqslant k < \ell \leqslant \dm -1} 
    ( \mu^{}_{\ell} - \mu^{}_{k}) \, \ne \, 0 \ts ,
\]
because $A$ has simple spectrum and only $\mu^{}_{0}=0$.  So, we can
extract the $\alpha^{}_{i}$ by applying $V^{-1}$ to the vector of
logarithms; see \cite{BS2} for an explicit formula for $V^{-1}$ in the
version needed here.

Having calculated $L$, which is the only real logarithm of $M$ under
the assumptions of Lemma~\ref{lem:real-log-exists}, we get the
following general result.

\begin{theorem}\label{thm:real-log-and-gen}
  Under the assumptions of Lemma~$\ref{lem:real-log-exists}$, the
  Markov matrix\/ $M$ has the unique real logarithm\/ $L$ as given in
  \eqref{eq:real-log}, with the coefficients\/ $\alpha^{}_{i}\in\RR$
  as derived above.
   
  Then, $M$ is embeddable if and only if the matrix\/ $L$ is a Markov
  generator.  \qed
\end{theorem}

\begin{remark}\label{rem:cyclic-1}
  The uniqueness result for $\log (M)$ also applies to \emph{cyclic}
  matrices, which are the matrices for which the characteristic
  polynomial is also the minimal polynomial; they are called
  \emph{non-derogatory} in the matrix analysis literature
  \cite{HJ,Higham}.  Cyclic matrices may contain non-trivial Jordan
  blocks; see \cite{interval} for concrete examples in our context
  where this occurs. As explained in \cite{BS2}, one can then still
  calculate the $\alpha^{}_{i}$ from a linear system of equations.
  Solving the latter requires the technically more involved confluent
  version of the Vandermonde matrix; see \cite[Thm.~5.3]{BS2} and the
  references given in its proof.  Further, as stated in
  Fact~\ref{fact:Culver}, uniqueness also holds when
  $\sigma (M) \subset \RR_{+}$ and no elementary Jordan block occurs
  more than once, which is yet slightly more general than being
  cyclic, and needed later.  \exend
\end{remark}

Let us now become a bit more specific on what kind of upper-triangular
matrices we will have to consider for the recombination process, and
how partitions enter the picture.

\subsection{Partitions}

Let $S$ be a finite set, and consider the lattice $\cP (S)$ of
partitions of $S$; see \cite{Aigner} for background material and
\cite{fast} and references therein for details of the present setting.
Here, we write a partition of $S$ as
$\cA = \{ A_{1}, \dots , A_{m} \}$, where $m = |\cA|$ is the number of
its (non-empty) parts (also called blocks), and one has
$A_{i} \cap A_{j} = \varnothing$ for all $i\ne j$ together with
$A_{1} \cup \dots \cup A_{m} = S$. The natural ordering relation is
denoted by $\preccurlyeq$, where $\cA \preccurlyeq \cB$ means that
$\cA$ is \emph{finer} than $\cB$, or that $\cB$ is \emph{coarser} than
$\cA$.  The conditions $\cA \preccurlyeq \cB$ and
$\cB \succcurlyeq \cA$ are synonymous, while $\cA \prec\cB$ means
$\cA \preccurlyeq \cB$ together with $\cA \ne \cB$, so $\cA$ is
strictly finer than $\cB$.

The joint refinement (or meet) of two partitions $\cA$ and $\cB$ is
written as $\cA \wedge \cB$, and is the coarsest partition below $\cA$
and $\cB$.  The unique \emph{minimal} partition within the lattice
$\cP (S)$ is denoted as
$\pmin = \pmin^{\ts S} = \big\{ \{x\} \mid x \in S \big\}$, and the
unique \emph{maximal} one as $\pmax = \pmax^{\nts S} = \{ S \}$.  When
$U$ and $V$ are disjoint sets, both finite and non-empty, two
partitions $\cA\in\ts\cP(U)$ and $\cB\in\ts\cP(V)$ can be joined (in
the obvious way) to form an element of $\cP(U\nts \cup V)$. We denote
such a \emph{joining} by $\cA\sqcup \cB$, which is meant to indicate
the different roles of $\cP (U)$, $\cP(V)$ and $\cP(U\nts \cup V)$,
and similarly for multiple joinings.  Conversely, if
$\varnothing\ne U\nts\subseteq S$, a partition $\cA\in\ts\cP(S)$, with
$\cA = \{ A_{1}, \dots , A_{m} \}$ say, defines a unique partition of
$U$ by restriction. The latter is denoted by $\cA|^{\pa}_{U}$, and its
parts are precisely all non-empty sets of the form $A_{i} \cap U$ with
$1\leqslant i \leqslant m$, so all $A_i$ with
$A_i \cap U = \varnothing$ are discarded. For $U\subseteq S$, the
maximal partition in $\cP (U)$ is $\{ U \}\eqdef \pmax^{\nts U}$.

There is one technical identity on the joining of partitions that we
shall need. If $A$ is a part of $\cA$, we write $\cA\sm A$ instead of
$\cA\sm\{ A \}$ for simplicity. To continue, we also need integer
linear combinations of partitions, considered as formal sums, and
their (obvious) distributive behaviour in joinings with another
partition.

\begin{lemma}\label{lem:technical}
  Let\/ $\cA , \cB , \cC \in \cP (S)$ be arbitrary, but fixed.  Then,
  for all\/ $A\in\cA$, one has
\[
\begin{split}
  \sum_{B\in\cB} & \bigl[ (\cA\sm A) \sqcup \bigl( (\cB \sm B) \sqcup
     \cC |_B \bigr) |_A - (\cA\sm A) \sqcup \cB |_A \bigr] \\
  & = \sum_{B' \in \cB |_A} \bigl[ (\cA\sm A) \sqcup (\cB |_A \sm B') 
     \sqcup \cC |_{B'} - (\cA \sm A) \sqcup \cB |_A \bigr] .
\end{split}  
\]
\end{lemma}

\begin{proof}
  If $\cC = \pmax$, one has $(\cB \sm B) \sqcup \cC |_B = \cB$, and
  the left-hand side vanishes, as does the right-hand side, here due
  to $\cC |_{B'} = B'$. Likewise, $\cB = \pmax$ means $\cB = \{ S \}$,
  so the left-hand side simplifies to
  $(\cA\sm A) \sqcup (\cC |_A - \cB |_A)$. As this case also means
  $\cB |_A = \{ A \}$, the right-hand side only has the summand with
  $B'=A$, which gives the same expression.

  Fix now $A\in\cA$, and let $\cB = \{ B_{1}, \ldots , B_{k} \}$ and
  $\cC = \{ C_{1}, \ldots , C_{\ell} \}$, where $k,\ell \geqslant 2$.
  For any $1\leqslant i \leqslant k$, we also consider
  $\cB \sm B_i = \{ B_{1}, \ldots , \widehat{B_i}, \ldots , B_{k} \}$,
  where $\widehat{\, . \,}$ indicates the missing part. Then,
  $\cC |_{B_i} = \{ C_{1} \cap B_{i}, \ldots , C_{\ell} \cap B_{i} \}$
  with the understanding that empty sets are omitted.

Now, we get 
\[
  \bigl( (\cB \sm B_i) \sqcup \cC |_{B_i} \bigr) |_A \, = \, \{ B_{1}
  \cap A, \ldots , \widehat{B_i \cap A}, \ldots , B_k \cap A \} \cup
  \{ C_j \cap B_i \cap A : 1 \leqslant j \leqslant \ell \} \ts ,
\]   
again removing all empty sets.  If $B_i \cap A = \varnothing$, it is
clear that $\cC |_{B_i}$ is disjoint from $A$, hence
$\bigl( (\cB\sm B_i)\sqcup \cC |_{B_i} \bigr) \big|_A =\cB |_A$, and
this term does not contribute to the sum on the left-hand side of our
claim, as the difference in the square brackets vanishes.  This means
that the sum effectively runs over $B\in\cB$ subject to the condition
that $B\cap A \ne \varnothing$.

On the other side, we have
$\cB |_A = \{ B_1 \cap A, \ldots , B_k \cap A \}$, again discarding
empty sets, hence
\[
  \bigl( \cB |_A \sm (B_i \cap A) \bigr) \sqcup \cC |_{B_i \cap A} \,
  = \, \{ B_1 \cap A, \ldots , \widehat{B_i \cap A}, \ldots , B_k \cap
  A \} \cup \{ C_j \cap B_i \cap A : 1 \leqslant j \leqslant \ell \}
  \ts ,
\]
which agrees with the previous expression when
$B_i \cap A \ne \varnothing$.  Since the second sum only runs over
such parts, we see that both sums contain the same terms, and are thus
equal.
\end{proof}

In the identity in Lemma~\ref{lem:technical}, the partition $\cA\sm A$
can be taken out of the formal sum on both sides, in the obvious way,
which leads to the following variant.

\begin{coro}\label{coro:technical}
  Let\/ $\varnothing\ne A\subseteq S$ be arbitrary, but fixed. Then,
  one has the identity
\[
  \sum_{B\in\cB} \bigl[ \bigl( (\cB \sm B) \sqcup \cC |_B \bigr) |_A -
  \ts \cB |_A \bigr] \, = \! \sum_{B' \in \cB |_A} \! \bigl[ (\cB |_A
  \sm B') \sqcup \cC |_{B'} - \ts \cB |_A \bigr] ,
\]   
   which holds for all\/ $\cB, \cC \in \cP (S)$.  \qed
\end{coro}

Let us now see how partitions become useful for the stochastic
processes we are after.

\section{Partitioning processes for recombination}\label{sec:reco-part}

Let $S=\{1, 2, \ldots , n\}$ be the set of sites (typically
representing sequence positions or genes). There are $\bell_n$
elements in $\cP (S)$, where $\bell_n$ is the $n$-th Bell number, with
generating function
\[
    \sum_{n=0}^{\infty} \myfrac{\bell_n}{n {\ts} !} \ts z^{n} \ts = \,
    \ee^{\ee^z - 1} \ts = \,
       1 + z + 2 z^2 + 5 z^3 + 15 z^4 + 52 z^5 + \ldots \ts .
\]
The Bell numbers satisfy
$\bell_{n+1} = \sum_{m=0}^{n} \binom{n}{m} \bell_m$; see
\cite[\texttt{A{\ts}000{\ts}110}]{OEIS} for more. Let us now introduce
the recombination matrix classes of Markov matrices and generators for
dimension $\dm = \bell_n$, where it is natural to use the partitions
of $S$ for indexing the $\dm$ directions in $\RR^{\!\dm}\!$.

\subsection{Markov matrices}

To describe the recombination process in discrete time, we know from
\cite{fast,EMB} that
$M= \bigl( M^{}_{\!\cA \ts \cB} \bigr)_{\cA, \cB \in \cP (S)}$ with
$M^{}_{\!\cA \ts\cB} = 0$ whenever $\cB \not \preccurlyeq \cA$, which
means that $M$ is upper triangular. In fact, $M$ is an element of the
incidence algebra of $\cP (S)$, compare \cite{Spiegel}, which is an
even stronger property. The transition probabilities
$r^{}_{\! \cA} \in [0,1]$ from $\pmax$ to $\cA\in\cP(S)$ constitute the
first row of $M$, and we write
$M^{}_{\pmax \cA} = r(\cA) = r^{}_{\! \cA}$. Note that the $r(\cA)$
with $\cA \in \cP (S)$ form a probability (row) vector. All other
elements of $M$ derive from here via \emph{marginalisation}, as given
by
\begin{equation}\label{eq:M-marg}
    M^{}_{\!\cA\ts \cB} \, = \, \begin{cases}
    \prod_{A\in\cA} r^{\ts A}_{\nts \cB |_A}  ,  & 
         \text{if } \cB \preccurlyeq \cA  \ts , \\
    0 \ts , & \text{otherwise} \ts ,  \end{cases}
\end{equation}
with the marginal probabilities
\begin{equation}\label{eq:marg-prob}
    r^{\ts A}_{\cD} \, \defeq \sum_{\substack{\cC \in \cP (S) \\
         \cC |_A = \cD}} r (\cC) \ts ,
\end{equation}
for any $\cD \in \cP (A)$.  It is easy to see that $M=\one$ satisfies
\eqref{eq:M-marg}, with $r (\cC) =\delta_{\pmax \cC}$.  In general,
the matrix elements of $M$ depend on the $r (\cA)$ in a
\emph{non-linear} manner once $\cA$ has more than one non-singleton
part. The eigenvalues of $M$ are its diagonal elements, so
\begin{equation}\label{eq:eigen-M}
  \lambda^{}_{\cA} \, = \, M^{}_{\! \cA \ts \cA} \ts = \prod_{A\in\cA}
  r^{\ts A}_{\nts \{ \nts A\}} \, = \prod_{A\in\cA}
     \sum_{\substack{\cC \in \cP (S) \\
      \cC |_A = \{ \nts A \} }} r (\cC) \ts ,
\end{equation}
using $\cA |_A = \{ \nts A \} = \pmax^{\nts A}$. The eigenvalues are
thus generally non-linear in the parameters $r (\cA)$.

\begin{remark}\label{rem:non-singular}
  One has $\det(M) = 0$ when $r(\pmax) = 0$, because
  $\lambda_{\pmax} = r(\pmax)$. In fact, since $r$ is a probability
  vector and $r(\pmax)$ occurs as a summand in each of the sums in
  \eqref{eq:eigen-M}, the triangular matrix $M$ is singular if and
  only if $r(\pmax) = 0$, and has positive spectrum otherwise.
\exend
\end{remark}

Let us next derive that, for fixed $n$ and $\dm =\bell_n$, the set of
Markov matrices of the form \eqref{eq:M-marg} is a semigroup
(actually, a monoid) under matrix multiplication.  Consider matrices
$M$ and $\tiM$ with defining vectors $r$ and $\tir$, so
$M^{}_{\pmax \cA} = r^{}_{\! \cA}$ and
$\tiM^{}_{\pmax \cA} = \tir^{}_{\! \cA}$ for $\cA\in\cP (S)$ together
with
\[
    M^{}_{\!\cA \ts \cB} \, = \prod_{A\in\cA} r^{\ts A}_{\cB |_A}
    \quad \text{and} \quad  \tiM^{}_{\! \cA \ts \cB} \, =
    \prod_{A\in\cA} \tir^{\ts A}_{\cB |_A}
\]
for $\cB \preccurlyeq \cA$, while all other matrix elements
vanish. Note that these relations, for $\cA=\pmax$ and arbitrary
$\cB$, simply reduce to those for the first row as given in the
preceding line. Further, for $\cB = \cA$, they give the diagonal
elements and thus the eigenvalues of the matrices. Now, consider a
product, $\whM = \tiM M$. Its matrix elements for
$\cC \preccurlyeq \cA$ satisfy
\[
    \whM^{}_{\! \cA \ts \cC} \, = \!
    \sum_{\cC \preccurlyeq \udo{\cB} \preccurlyeq \cA}
    \! \tiM^{}_{\! \cA \ts \cB} \, M^{}_{\cB \ts \cC} \, = \!
    \sum_{\cC \preccurlyeq \udo{\cB} \preccurlyeq \cA}
    \, \prod_{A\in\cA} \tir^{\ts A}_{\cB |_A}
    \prod_{B\in\cB} r^{\ts B}_{\cC |_B} \ts ,
\]
where the underdot marks the summation variable, while all other
elements vanish due to the special triangular structure of $M$ and
$\tiM$. In particular, the first row of $\whM$ has the entries
\begin{equation}\label{eq:product-r}
   \whr^{}_{\cC} \, \defeq \,\whM^{}_{\pmax \cC} \, = \!
   \sum_{\cC \preccurlyeq \udo{\cB} \preccurlyeq\pmax} \! \tir^{}_{\cB} 
   \prod_{B\in\cB} r^{\ts B}_{\cC |_B} \, \eqdef \,
   \bigl( \tir \bd r \bigr) (\cC) \ts .
\end{equation}
To establish that $\whM$ is again a Markov matrix of the right type,
we have to show that
$\whM^{}_{\! \cA \ts \cC} = \prod_{A\in\cA} \whr^{\ts A}_{\cC |_A}$
holds for $\whM$ for all $\cC\preccurlyeq\cA$, with the parameters
$\whr^{}_{\cC}$ from \eqref{eq:product-r}. In other words, using
\eqref{eq:product-r}, we have to prove that
\begin{equation}\label{eq:need-to-show}
    \sum_{\cC \preccurlyeq \udo{\cB} \preccurlyeq \cA} \,
    \prod_{A\in\cA}  \tir^{\ts A}_{\cB |_A} \prod_{B\in\cB} r^{\ts B}_{\cC |_B}
    \, = \prod_{A\in\cA} \whr^{\ts A}_{\cC |_A}
\end{equation}
holds for all $\cA, \cC \in \cP (S)$ with $\cC \preccurlyeq \cA$. One
important step in this is the following.

\begin{lemma}\label{lem:part-help}
  For any\/ $\varnothing \ne U\subseteq S$ and then every\/ 
  $\cA \in \cP (U)$, one has the identity
\[
    \whr^{\ts U}_{\! \cA} \, = \!
    \sum_{\cA \preccurlyeq \ts\udo{\fa}\ts \preccurlyeq \pmax^{\nts U}}
    \! \tir^{U}_{\fa} \prod_{a \in \fa} r^{\, a}_{\cA |_{a}} \ts .
\]   
\end{lemma}

\begin{proof}
  Let $\varnothing \ne U\subseteq S$ and $\cA \in \cP (U)$ be fixed.
  Then, by the definition of the marginal probabilities from
  \eqref{eq:marg-prob} in conjunction with \eqref{eq:product-r}, we
  get
\begin{align*}
   \whr^{U}_{\cA} \,  = \! \sum_{\substack{\cB\in\cP(S) \\ \cB |_U = \cA}}
   \! \whr^{}_{\cB} \, = \! \sum_{\substack{\cB\in\cP(S) \\ \cB |_U = \cA}}
   \, \sum_{\cB \preccurlyeq \udo{\cC} \preccurlyeq \pmax}  \tir^{}_{\cC}
   \prod_{C\in\cC} r^{\ts C}_{\cB |_C}  \,  =  \!
   \sum_{\cA \preccurlyeq \ts \udo{\fa} \ts \preccurlyeq \pmax^{\nts U}}
   \, \sum_{\substack{\cC \in \cP (S) \\ \cC |_U = \ts \fa}}  \! \tir^{}_{\cC}
      \sum_{\substack{\udo{\cB} \preccurlyeq \cC \\ \cB |_U = \ts \cA}}
      \prod_{C\in\cC} r^{\ts C}_{\cB |_C} \ts ,
\end{align*}  
where the last step follows by a suitable resummation.
   
To establish our claim, we now need to evaluate the last sum (over
$\cB$).  For any given $\fa = \{ a^{}_{1}, \ldots , a^{}_{k} \}$ and
$\cC \in\cP (S)$ with $\cC |_U = \fa$, we employ a calculation from
\cite{BBS} and write
$\cC = \{ C^{}_{1} , \ldots , C^{}_{k}, C^{\prime}_{1}, \ldots ,
C^{\prime}_{\ell} \}$ with $C^{}_i \cap U = a^{}_i$,
$C^{\prime}_{j} \cap U = \varnothing$, and
$\ell = \lvert \cC \rvert - k$.  This gives
\begin{align*}
  \sum_{\substack{\udo{\cB} \preccurlyeq \cC \\ \cB |_U = \ts \cA}}
   \prod_{C\in\cC} r^{\ts C}_{\cB |_C} \, & =
    \sum_{\substack{\fc^{}_{1} \in \cP (C^{}_1) \\ \fc^{}_{1} |_{a^{}_1} 
        = \cA |^{\phantom{\hat{I}}}_{a^{}_1}}} \cdots 
    \sum_{\substack{\fc^{}_{k} \in \cP (C^{}_k) \\ \fc^{}_{k} |_{a^{}_k} 
        = \cA |^{\phantom{\hat{I}}}_{a^{}_k}}}    
    \sum_{\fc^{\prime}_{1} \in \cP (C^{\prime}_1)} \cdots      
    \sum_{\fc^{\prime}_{\ell} \in \cP (C^{\prime}_{\ell})}
    \, \prod_{i=1}^{k}   r^{C^{}_{i}}_{\fc^{\phantom{\prime}}_{i}}
    \prod_{j=1}^{\ell}  r^{C^{\prime}_{j}}_{\fc^{\prime}_{j}} \\
     & = \, \biggl( \, \prod_{i=1}^{k}
          \sum_{\substack{\fc^{}_{i} \in \cP (C^{}_{i}) \\[1mm]
          \fc^{}_{i} |^{}_{a_i} = \cA |_{a_i}}} 
          \! r^{C^{}_{i}}_{\fc^{\phantom{\prime}}_{i}}\biggr)
          \biggl( \, \prod_{j=1}^{\ell} \,
          \sum_{\fc^{\prime}_{j} \in \cP (C^{\prime}_{j})}
          \! r^{C^{\prime}_{j}}_{\fc^{\prime}_{j}} \biggr) \, = \,
          \prod_{i=1}^{k} r^{\, a^{}_{i}}_{\!\cA |_{a^{}_{i}}}  \ts ,      
\end{align*}   
where the last step uses that each factor in the first product is a
sum that gives one of the $r^{\, a^{}_{i}}_{\cA |_{a^{}_{i}}} \nts $,
while each factor in the second product is a sum that adds to
$1$. These relations use the fact that, for any
$\varnothing \ne W \subseteq V \subseteq S$ and every
$\cB \in \cP (S)$, one has
\[
   r^{\ts W}_{\cB |_W} \, = \sum_{\substack{\cD \in \cP (S) \\ 
        \cD |_W = \cB |_W}} \! r^{}_{\cD} \, =
        \sum_{\substack{\cE \in \cP (V) \\
        \cE |_W = \cB |_W}} \! r^{V}_{\cE}  ,
\]  
in analogy to \cite[Eq.~(26)]{BBS}. Putting the two pieces together
and observing the chosen representation of $\cA$ establishes our
claim.
\end{proof}

Now, we can verify \eqref{eq:need-to-show} as follows. Let
$\cA\in\cP(S)$ be fixed, and assume
$\cA = \{ A^{}_1, \ldots , A^{}_m \}$. Any $\cB \preccurlyeq\cA$ can
then be written as $\cB = \fb^{}_{1} \sqcup \ldots \sqcup \fb^{}_{m}$
with $\fb_i \in \cP (A_i)$, and we have $\cB |_{A_i} = \fb_i$. The
left-hand side of \eqref{eq:need-to-show} then is
\begin{align*}
   \sum_{\cC \preccurlyeq \udo{\cB} \preccurlyeq \cA}  \,
  \prod_{i=1}^{m}  \tir^{\ts A_i}_{\cB |_{A_i}}
       \prod_{B\in\cB} r^{\ts B}_{\cC |_B}
    & = \sum_{\cC |_{A_1} \nts \preccurlyeq  \ts \udo{\fb}^{}_{1}
       \preccurlyeq  \pmax^{\! A_1}} \cdots
       \sum_{\cC |_{A_m} \nts \preccurlyeq  \ts \udo{\fb}^{}_{m}
       \preccurlyeq \pmax^{\! A_m}} \,  \prod_{i=1}^{m} \tir^{A_i}_{\fb_i}
       \prod_{b_i \in\fb_i} r^{\ts b_i}_{\cC |_{b_i}} \\[1mm]
    & = \, \prod_{i=1}^{m} \, \sum_ {\cC |_{A_i} \nts \preccurlyeq  \ts 
       \udo{\fb}^{}_{i} \preccurlyeq \pmax^{\! A_i}} \! \! \tir^{A_i}_{\fb_i}
       \prod_{b_i \in \fb_i} \! r^{\ts b_i}_{\cC |_{b_i}}  \; = \,
       \prod_{i=1}^{m} \whr^{\ts A_i}_{\cC |_{A_i}} \ts ,
\end{align*}
where the final step follows from an application of
Lemma~\ref{lem:part-help} to each of the factors. Here, the last
product is nothing but the right-hand side of
\eqref{eq:need-to-show}. As $\cA\in\cP(S)$ was arbitrary, we can now
state the following result.

\begin{prop}\label{prop:Markov-closure}
  Let\/ $S = \{ 1, \ldots , n \}$ be fixed. Then, the set of all
  Markov matrices that satisfy the marginalisation structure from
  \eqref{eq:M-marg} contains\/ $\one$ and forms a monoid under matrix
  multiplication. It is also topologically closed, that is, closed
  under taking limits.
  
  Equivalently, the simplex of probability vectors in\/ $\RR^{\dm}$
  with\/ $\dm = \bell_n$ is a monoid under the multiplication\/ $\bd$
  defined in \eqref{eq:product-r}, with the parametrisation of the
  Markov matrices by the probability vectors defining an isomorphism.
\end{prop}

\begin{proof}
  The claim on the Markov matrices follows from our above derivations.
  The parametrisation of $M$ through its first row in conjunction with
  \eqref{eq:M-marg} and \eqref{eq:marg-prob} clearly is a bijection
  between the recombination Markov matrices and the $\dm$-simplex of
  probability vectors. The homomorphism property follows from
  \eqref{eq:product-r}, while associativity of $\bd$, meaning
  $(r\bd s)\bd t = r \bd (s\bd t)$, is inherited from that of matrix
  multiplication, or can be verified via \eqref{eq:need-to-show}.
\end{proof}

For this result, there was no need to distinguish singular from
non-singular matrices. It is clear that the non-singular ones form an
open subset that still is a monoid under matrix multiplication.
Embeddability will only be relevant for this subset.  Also, by the
results of \cite{JS}, one should expect some interesting Lie-algebraic
structure as well.

\subsection{Extension and algebraic properties}
Let us note that the above calculation around $\whr$ relies on $M$ and
$\tiM$ having unit row sums, but does not need non-negativity of the
matrix elements. Consequently, if $M$ and $\tiM$ are real matrices
with unit row sums that satisfy \eqref{eq:M-marg}, $\tiM M$ is again a
matrix of this type. This motivates the following notion.

\begin{definition}\label{def:RMS}
  Let\/ $S = \{ 1, 2, \ldots, n \}$ with\/ $n\in\NN$ be fixed and
  set\/ $\dm=\bell_n$. An upper-triangular, real\/
  $\dm{\times}\dm$-matrix\/ $M$ is said to have a \emph{recombination
    matrix structure}, or RMS for short, if all row sums are\/ $1$ and
  if all elements of\/ $M$ follow from its first row according to
  \eqref{eq:M-marg} and \eqref{eq:marg-prob}.  The first row of\/ $M$,
  denoted by\/ $r$, is its \emph{parameter vector}.

  If\/ $r$ is a probability vector, $M$ is an RMS \emph{Markov}
  matrix.
\end{definition}

The algebraic closure condition under multiplication can be
interpreted as follows.  If we write an RMS Markov matrix as
$M = M(r)$, thus referring to the parametrisation of $M$ via the row
probability vector $r$, Eq.~\eqref{eq:product-r} implies the relation
\[
  M(r) M(r' \ts ) \, = \, M \bigl( r M(r' \ts )\bigr)
  \, = \, M( r \bd r' \ts ) \ts ,
\]
where $r M(r')$ is again a probability vector.

More generally, the parameter vectors of RMS matrices are arbitrary
elements of the hyperplane in $\RR^{\dm}$ that is perpendicular to
$(1,1,\ldots, 1)$ and contains the unit vector $e^{}_{\pmax}$. The
simplex of probability vectors is the convex hull of the $\bell_n$
unit vectors $e^{}_{\!\cA}$ with $\cA\in\cP (S)$. If one demands
invertibility, one also has the following.

\begin{prop}\label{prop:inverse}
  Let\/ $S = \{ 1,2, \ldots , n\}$ be fixed, set\/ $\dm=\bell_n$, and
  let\/ $M$ be a non-singular RMS matrix. Then, also\/ $M^{-1}$ is an
  RMS matrix, and the set of all non-singular RMS matrices forms a
  group under matrix multiplication.

  Further, the subset of all RMS matrices with positive spectrum 
  forms a subgroup.
\end{prop}

\begin{proof}
  Let a non-singular RMS matrix $M$ be given, with parameter vector
  $r \in \RR^{\!\dm}$ for $\dm = \bell_n$, where
  $\sum_{\cA \in \cP (S)} r(\cA) = 1$. We first observe that there is
  precisely one $\tir\in\RR^{\!\dm}$ such that
  $\whr (\cA) = \delta^{}_{\pmax \cA}$ holds in \eqref{eq:product-r}
  for every $\cA\in\cP(S)$. This is easily seen recursively, which
  gives
\begin{equation}\label{eq:r-inv}
  \tir (\cA) \, = \, \myfrac{\delta^{}_{\pmax \cA}}{ r(\pmax)}
  \, - \lambda^{-1}_{\!\cA} \sum_{\udo{\cB} \succ \cA} \tir (\cB)
  \prod_{B\in\cB} \sum_{\substack{\cC \in \cP (S) \\ \cC |_B = \cA |_B}}
  \! r (\cC) \ts ,
\end{equation}
where the $\lambda^{}_{\!\cA} \ne 0$ are the eigenvalues of $M$ from
\eqref{eq:eigen-M}, with $\lambda_{\pmax} = r(\pmax)$. This vector
satisfies $\sum_{\cA\in\cP(S)} \tir (\cA) = 1$, as one can check with
a calculation that is completely analogous to the one used in the
proof of Lemma~\ref{lem:part-help} (and thus omitted here).

Let $\tiM$ be the upper-triangular matrix defined by $\tir$ according
to \eqref{eq:M-marg}. Then, by the multiplicative closure, which
generalises from RMS Markov matrices to general RMS matrices by the
same calculation as used above, $\tiM M$ is an RMS matrix with
parameter vector $\delta = (\delta^{}_{\pmax \cA})^{}_{\cA\in\cP(S)}$,
which is the identity, so $\tiM M = \one$ and $\tiM = M^{-1}$ is an
RMS matrix as well.  The claimed group property is then clear.

When $M$ has positive spectrum, this is also true of its inverse,
which is clear from its triangular structure and implies the stated
subgroup property.
\end{proof}

Let us look at this in a different way, which is also an extension of
Remark~\ref{rem:non-singular}.

\begin{definition}\label{def:spec}
  Let $r\in \RR^{\dm}$ be a general row vector with row sum $1$.
  Then, we call the set $\{ \lambda^{}_{\!\cA} : \cA \in \cP(S) \}$
  with $\lambda^{}_{\!\cA} = \prod_{A\in\cA}\phi^{}_{\! A}$ the
  \emph{spectrum} of $r$, where the
\[
  \phi^{}_{U} \, \defeq \, r^{\, U}_{\! \{ U \} } \, = \!
  \sum_{\udo{\cA} |_U = \{ U \}} \! r (\cA) \qquad \text{with }
  \, \varnothing \ne U \subseteq S
\]
are the \emph{characteristic factors} of $r$. Further, $r$ has
\emph{simple spectrum} when the $\lambda^{}_{\!\cA}$ are distinct, $r$
is \emph{non-singular} when no $\lambda^{}_{\!\cA}$ vanishes, and $r$
has \emph{positive spectrum} when all $\lambda^{}_{\!\cA} > 0$.
\end{definition}

Note that $\lambda_{\pmax} = \phi^{}_{S}=r(\pmax)$, and $r$ being
non-singular is equivalent to $\phi^{}_{U} \ne 0$ for all
$\varnothing \ne U \subseteq S$. Note further that $\phi^{}_{U}$ is
an eigenvalue of the marginal Markov matrix defined by the
parameter vector $r^{U} \defeq (r^{U}_{\cD})^{}_{\cD \in \cP (U)}$
with entries according to \eqref{eq:marg-prob}. This matrix describes
the recombination process on the subsystem defined by $U$; see
\cite[Sec.~6]{fast}. The spectrum of $r$ is positive if
and only if all $\phi^{}_{U} > 0$. While one direction of this is
obvious, the other is a simple consequence of $\phi^{}_{\{i\}} = 1$
for any singleton set $U = \{i\}$. Indeed, this direction is clear for
$U=S$, while any other $U$ can be augmented by singleton sets to form
a partition of $S$. Now, Proposition~\ref{prop:inverse} has the
following consequence.

\begin{coro}\label{coro:r-inverse}
  Any non-singular, unit row sum vector\/ $r\in\RR^{\dm}$ has an
  inverse for\/ $\bd$, which is the vector\/ $\tir$ from
  \eqref{eq:r-inv}, and the characteristic factors are related by\/
  $\tilde{\phi}^{}_{U} = 1/\phi^{}_{U}$, for all\/
  $\varnothing \ne U \subseteq S$. In particular, if\/ $r$ has
  positive spectrum, then so does\/ $\tir$.
\end{coro}

\begin{proof}
  Since $\tir \bd r = \delta$, where
  $\delta(\cA) = \delta^{}_{\pmax\cA}$, and since the spectrum of
  $\delta$ consists of $\dm$ copies of $1$, we get
  $\tilde{\lambda}^{}_{\!\cA} \lambda^{}_{\cA} = 1$ for all
  $\cA\in\cP(S)$ from \eqref{eq:need-to-show} for $\cC = \cA$. This
  gives
  $\tilde{\phi}^{}_{S} = \tilde{\lambda}_{\pmax} = 1/\lambda_{\pmax} =
  1/\phi^{}_{S}$.  For any $\varnothing \ne U \subset S$, we form the
  partition
  $\cA^{}_{U} \defeq \{ U \} \sqcup \bigl\{ \{ i \} : i \in S\ts \sm U
  \bigr\}$ and obtain
\[
  \tilde{\phi}^{}_{\ts U} \, = \prod_{A\in\cA^{}_U} \tilde{\phi}^{}_{\nts A}
  \, = \, \tilde{\lambda}^{}_{\!\cA^{}_U} \, = \, \lambda^{-1}_{\!\cA^{}_U}
  \, = \prod_{A\in\cA^{}_U} \phi^{-1}_{A} \, = \, \phi^{-1}_{\ts U} \ts ,
\]
where we have used that the characteristic factors for singleton sets
are always $1$.

The claim on the positivity follows from this, too.
\end{proof}

The product $\bd$ from \eqref{eq:product-r} is complicated, but has
one important extra property. The set of vectors with unit row sum is
closed under convex combinations, and our product respects this in the
first argument, meaning that we have
\begin{equation}\label{eq:convex}
  \bigl( \alpha \ts r + (1 \nts -\alpha) s \bigr) \bd \ts t
  \, = \, \alpha \, r \bd \ts t + (1 \nts -\alpha) \ts s \bd \ts t
\end{equation}
for all $\alpha \in [0,1]$, as one can easily verify. So, although we
have no addition at our disposal, we can use this property to show the
following result on the existence of principal roots, where we use the
shorthand $w^2 \defeq w \bd w$ and similarly for other powers.

\begin{lemma}\label{lem:roots}
  Let\/ $r \in \RR^{\dm}$ be an arbitrary row vector with unit sum and
  positive spectrum. Then, for any integer\/ $p\geqslant 2$, there is
  a unit row sum vector\/ $w$ with positive spectrum that satisfies\/
  $w^{\ts p} = r$. When the spectrum of\/ $r$ is also simple, $w$ is
  unique.
\end{lemma}

\begin{proof}
  Inspired by the methods from \cite[Chs.~6 and 7]{Higham}, we employ
  a convergent Newton-type algorithm. Let $p\geqslant 2$ be fixed, set
  $w^{}_{0} = \delta$, and define the iteration
\[
  w^{}_{m+1} \, = \, \tfrac{p-1}{p}\ts w^{}_{m} +
    \tfrac{1}{p} \ts r \bd \ts w^{1-p}_{m}
\]
for $m\geqslant 0$, where $x^{1-p} = (x^{\ts p-1})^{-1}$ is well
defined for any non-singular row vector with unit sum.  Note that the
right-hand side is a convex combination of unit sum vectors and
thus again a vector of this type.  Since $w^{}_{0}=\delta$ has
positive spectrum, and since having positive spectrum is preserved
under taking inverses by Corollary~\ref{coro:r-inverse}, it is clear
(by induction) that $w_m$ has positive spectrum for every $m\in\NN$ as
well.

The sequence $(w_m)^{}_{m\in\NN}$ converges (by standard arguments,
using a suitable vector norm), where the unit sum is preserved. In
the limit, one has
$w = \frac{p-1}{p} \ts w + \frac{1}{p} \ts r \bd \ts w^{1-p}$, which
gives $w^{\ts p} = \frac{p-1}{p} \ts w^{\ts p} + \frac{1}{p}\ts r$
upon right multiplication with $w^{\ts p-1}$ and an application of
\eqref{eq:convex}, and thus $w^{\ts p} = r$. Now, for each
$\varnothing \ne U \subseteq S$, the corresponding characteristic
factors of the $w_m$ satisfy a scalar Newton iteration of the same
form, all starting from $1$ and staying within $\RR_+$, so each such
sequence must converge to a non-negative real number, which is then
actually positive as well. This shows that $w$ is a $p\ts$-th root of
$r$ with positive spectrum.

For the claimed uniqueness, observe first that the unit sum vector
$r$ defines an RMS matrix with simple, positive spectrum, which is
diagonalisable and possesses a unique $p\ts$-th matrix root with
positive spectrum that is simple as well; see \cite[Thm.~7.2]{Higham}.
Since each $w$ from above defines an RMS matrix with positive spectrum
that is a $p\ts$-th matrix root of $M$ as well, only one such $w$ can
exist.
\end{proof}

\begin{remark}\label{rem:non-pos}
  When $r$ is a probability vector, it has positive spectrum if and
  only if $r(\pmax)>0$. For any integer $p\geqslant 2$, there is then
  a $p\ts$-th root with positive spectrum, but this need not be a
  probability vector.  In fact, Kingman's theorem \cite[Prop.~7]{King}
  implies that $M=M(r)$ is embeddable if and only if it has a
  $p\ts$-th Markov root for every $p\in\NN$. In other words, a
  probability vector $r$ parametrises an embeddable RMS Markov matrix
  if and only if, for every $p\in\NN$, it possesses a $p\ts$-th root
  in the sense of Lemma~\ref{lem:roots} that is a probability vector,
  meaning that it is infinitely divisible with respect to the product
  defined by $\bd$. Unfortunately, this does not lead to a practically
  useful criterion.  \exend
\end{remark}

The non-singular, upper-triangular matrices form a well-known Lie
group, where those with positive spectrum form a subgroup. Our above
results show that the RMS matrices with positive spectrum form yet
another subgroup, which contains all non-singular RMS Markov matrices.
We should thus expect that some interesting connections emerge via the
corresponding Lie algebra. Let us now develop this picture step by
step.

\subsection{Markov generators}

In continuous time, the recombination Markov generators $Q$ have a
related structure \cite{fast,EMB}. The free parameters are the entries
$Q^{}_{\pmax \cA} = \rho (\cA) = \rho^{}_{\! \cA}$ with
$\pmax \ne \cA \in \cP(S)$, which are the transition \emph{rates} from
$\pmax$ to $\cA$, while
$Q^{}_{\pmax \pmax} = - \sum_{\pmax\ne \cA\in\cP(S)} Q^{}_{\pmax \cA}$
ensures row sum $0$. To specify the other matrix entries, we introduce
the relation $\cB \klein \cA$ for the case that $\cB\ne\cA$ refines
precisely one part of $\cA$. Then, for $\cB \prec \cA$, we get from
\cite{fast,EMB} that
\begin{equation}\label{eq:Q-marg}
     Q^{}_{\! \cA \ts \cB} \, = \, \begin{cases}
     \rho^{A}_{\ts \cC} \ts , & \text{if } \cB \klein \cA \text{ with }
     \cC = \cB |_A  \ne \{ A \} \ts ,
        \\  0 \ts , & \text{otherwise} \ts , \end{cases}
\end{equation}
with the marginal rates
\begin{equation}\label{eq:marg-rates}
    \rho^{A}_{\ts \cC} \, = 
    \sum_{\substack{\cD \in \cP (S) \\ \cD |_A = \cC}} \rho(\cD)
\end{equation}
for all $\cC \in \cP (A)$ with
$\cC \ne \{ \nts A \} = \pmax^{\nts A}$. One can verify that all row
sums of $Q$ are $0$, as a consequence of this property for $\rho$.
The diagonal entries of $Q$ are again its eigenvalues,
\begin{equation}\label{eq:eigen-Q}
   \mu^{}_{\cA} \, = \, Q^{}_{\! \cA \ts \cA} \, = \, - \sum_{\cB \prec \cA}
   Q^{}_{\! \cA \ts \cB} \, = \, - \!\sum_{\cB\in\cP(S)}\! \rho (\cB) \,
   \card \bigl\{ A\in\cA : \cB |_A \ne \{ \nts A \} \bigr\} ,
\end{equation}
with $\cA \in \cP(S)$, where the last representation follows from a
simple combinatorial calculation. In particular, one has
$\mu^{}_{\pmax} = 0$.  Note that the eigenvalues of $Q$ are
\emph{linear} in the parameters. It is now natural to introduce the
following notion.

\begin{definition}
  Let\/ $S = \{ 1, 2, \ldots, n\}$ with\/ $n\in\NN$ be fixed and set\/
  $\dm = \bell_n$. A real, upper-triangular matrix\/ $Q$ is said to
  have \emph{recombination rate structure}, or RRS for short, if it
  has zero row sums and if all entries of\/ $Q$ emerge from its first
  row via \eqref{eq:Q-marg} and \eqref{eq:marg-rates}. Its first row,
  denoted by\/ $\rho$, it its \emph{parameter vector}.
  
  If, in addition, $Q$ is a Markov generator, it is called an RRS 
  Markov generator.
\end{definition}

It is not obvious why RRS matrices form the correct counterpart to
non-singular RMS matrices with positive spectrum. For Markov matrices,
this follows from the probabilistic derivation in
\cite{fast,EMB}. Here, we give an independent argument as
follows. Assume that a one-parameter family
$\{ M(t) : 0 \leqslant t < \varepsilon \}$ of RMS
matrices\footnote{The time dependence for this short argument is
  written as $M(t)$, which should be distinguished from the parameter
  dependence $M(r)$ used elsewhere.} is given, for some
$\varepsilon > 0$, with $M(0)=\one$ and differentiability near
$0$. So, $M(t)$ has parameters $r(\cA, t)$ for $\cA\in\cP (S)$ and
small $t\geqslant 0$, with $r (\cA, 0) = \delta_{\nts\cA, \pmax}$ and
$\rho (\cA) \defeq \dot{r} (\cA , 0)$, where we use $\dot{r}$ for the
time derivative of $r$, and similarly for other quantities.

Now, setting
$Q^{}_{\! \cA \ts \cB} \defeq \dot{M}^{}_{\! \cA \ts \cB} (0)$, it is
clear that $Q^{}_{\! \cA \ts \cB} =0$ whenever
$\cB \not \preccurlyeq \cA$. Next, when $\cB \prec \cA$,
Eqs.~\eqref{eq:M-marg} and \eqref{eq:marg-prob} via the product rule
imply that
\[
    Q^{}_{\! \cA \ts \cB} \, \defeq \, \dot{M}^{}_{\! \cA \ts \cB} (0) \, = 
    \sum_{A\in\cA} \dot{r}^{\ts A}_{\cB |_A}  (0) \!
      \prod_{A' \in \cA \setminus A} \! r^{\ts A'}_{\cB |_{A'}} (0)
\]
with
\[    
   \dot{r}^{\ts A}_{\cB |_A} (0) \, = \! \sum_{\substack{\cC\in\cP (S) \\
       \cC |_A = \cB |_A}} \!  \rho (\cC) \qquad \text{and} \qquad
       r^{\ts A'}_{\cB |_{A'}} (0) \, = \,  \begin{cases} 1 \ts , &
       \text{if } \cB |_{A'} = \{ A' \} ,  \\ 0 \ts , &
       \text{otherwise} \ts .  \end{cases}
\]
The case distinction is an easy consequence of the marginalisation
formula in conjunction with $r (\cA, 0) = \delta_{\nts \cA ,
  \pmax}$. Since $\cB \prec \cA$ means that $\cB$ must split at least
one part of $\cA$, we only get a contribution when $\cB\klein\cA$, and
$0$ otherwise.

To determine the remaining elements $Q^{}_{\!\cA\ts\cA}$, we observe
that $\sum_{\cB\in\cP(S)} M^{}_{\!\cA\ts\cB} (t) = 1$ implies the row
sum condition
$\sum_{\cB\in\cP(S)} Q^{}_{\!\cA\ts\cB} = \sum_{\cB\in\cP(S)}
\dot{M}^{}_{\!\cA\ts\cB} (0) = 0$ and thus, via upper triangularity,
the relation
$Q^{}_{\!\cA\ts\cA} = -\sum_{\cB\prec\cA} Q^{}_{\!\cA\ts\cB}$.
Putting the pieces together, we see that the tangent to $M(t)$ at $0$
must be an RRS matrix.  When $M(t)$ is Markov, as assumed, all
$\rho (\cA)$ with $\cA\ne\pmax$ must be non-negative, because
$M(0)=\one$ and the parameter vector of $M(t)$ cannot get negative
entries with increasing $t$. So, $Q$ is then a Markov generator, due
to \eqref{eq:marg-rates}. When we allow $M(t)$ to be a general,
non-singular RMS matrix, the tangent element still satisfies
Eqs.~\eqref{eq:Q-marg} and \eqref{eq:eigen-Q} with
\eqref{eq:marg-rates}, so is RRS, but need no longer be a Markov
generator. In particular, the parameter vector has zero row sum, but
the $\rho (\cA)$ need not be non-negative for $\cA\ne\pmax$.

If we now assume that a non-singular RMS Markov matrix $M$ has a real
logarithm $R$, we may consider $M(t) = \ee^{t R}$ with $M(0)=\one$ and
$R = \dot{M} (0)$. Note that $M(t)$ for $t\ne 1$ need not be
Markov. Now, let $M$ have simple spectrum. Then, $\ee^{t R}$ is RMS
for all $t \geqslant 0$, which can be seen as follows. Clearly, $M$
has distinct, positive eigenvalues, and is upper triangular. Then, for
any $m\geqslant 2$, it has a unique $m$-th root in upper-triangular
form with positive eigenvalues, by an application of
\cite[Thm.~7.2]{Higham}. It can easily be calculated as
$M^{1/m} = T D^{1/m} T^{-1}$ where $D$ is the diagonal of $M$ and $T$
is an invertible upper-triangular matrix that columnwise contains the
eigenvectors of $M$.

On the other hand, the parameter vector for $M$ has positive spectrum,
so Lemma~\ref{lem:roots} tells us that there is a unit sum vector
$w$ with positive spectrum such that $w^m = r$. This $w$ defines an
RMS matrix whose $m$-th power is $M$, due to multiplicative closure,
so defines an $m$-th root of $M$ with positive eigenvalues.  But there
is only one $m$-th root of $M$ with positive eigenvalues by
\cite[Thm.~7.2]{Higham}, because we assumed $M$ to have simple
spectrum, so they must agree, and this root must equal $\ee^{R/m}$.

Consequently, all rational powers of $\ee^R$ are of the right form,
hence also $\ee^{tR}$ for all $t\geqslant 0$, by a standard continuity
argument. Now, we can invoke the above calculations to see that the
derivative at $0$, which is $R$, has the correct form, and we may
conclude as follows.

\begin{prop}\label{prop:correct-form}
  Let\/ $M$ be a non-singular RMS Markov with simple, positive
  spectrum.  Then, it has a unique real matrix logarithm, $R$.
  Further, this\/ $R$ satisfies the linear conditions of
  Eq.~\eqref{eq:Q-marg}, and thus is a matrix of RRS type.

  More generally, this applies to any RMS matrix with simple,
  positive spectrum. \qed
\end{prop}

Since non-singular, upper-triangular Markov matrices must have
positive eigenvalues, the existence of a real logarithm is clear, and
the latter is unique when no elementary Jordan block is repeated.
Indeed, we have the following extension of
Theorem~\ref{thm:real-log-and-gen} and
Proposition~\ref{prop:correct-form}.

\begin{theorem}\label{thm:RMS-RRS}
  Let\/ $M$ be a non-singular RMS matrix with positive spectrum, which
  includes the case that\/ $M$ is an RMS Markov matrix with\/
  $\det(M) >0$. Then, $M$ has a real logarithm\/ $R$ of RRS type.
  
  Further, when no elementary Jordan block of the JNF of\/ $M$ over\/
  $\CC$ occurs more than once, $R$ is unique and upper
  triangular. If\/ $M$ is also Markov, it is then embeddable if and
  only if this\/ $R$ is a Markov generator.

  More generally, in the case of repeated Jordan blocks, where further
  real logarithms of\/ $M$ exist, no other one is upper triangular,
  hence not of RRS type.
\end{theorem}

\begin{proof}
  Assume first that $M$ has a real logarithm of RRS type.  Then, the
  uniqueness claim in the absence of repeated Jordan blocks is a
  consequence of Culver's theorem (Fact~\ref{fact:Culver}). In fact,
  the unique real matrix logarithm of $M$ must then be the principal
  matrix logarithm, which always exists under our assumption on
  positive spectrum. This can be seen via the series from
  Lemma~\ref{lem:real-log-exists} (when it converges, as it does for
  $M$ being Markov) or from the general integral formula
  \cite[Thm.~11.1]{Higham},
  $ L = \log (M) = \int_{0}^{1} \bigl( t M + (1-t) \one\bigr)^{-1}
  (M-\one) \, \mathrm{d} \ts t $.  Clearly, $L$ is real and inherits
  upper triangularity from $M$. When $M$ is also Markov, since we have
  $M=\ee^L$ and $L$ is unique, the embeddability claim is clear.

  The existence of a real logarithm of RRS type, $R$ say, for
  RMS matrices with simple, positive spectrum follows from
  Proposition~\ref{prop:correct-form}. As the principal matrix
  logarithm $L$ of $M$ is the only real logarithm in this case, which
  is also upper triangular, it must be this RRS matrix, so $R=L$.  
  When the positive spectrum of $M$ has degeneracies, it remains to
  show that $L$ still is an upper-triangular real logarithm of RRS
  type, and that it is the only one.
  
  If $M$ is an RMS matrix with positive spectrum, with some degeneracy
  say, any neighbourhood of $M$ within the RMS matrices will contain
  RMS matrices with simple spectrum. This immediately follows from
  the dependence of $M(r)$ and its eigenvalues on the parameter
  vector $r$ in Eq.~\eqref{eq:eigen-M}. So, by using sufficiently small
  neighbourhoods, we can make sure that the approximating matrices
  all have simple, positive spectrum. Each such 
  matrix has its principal matrix logarithm as its
  unique real logarithm, which is of RRS type. So, we can select a
  sequence of such matrices that converge to $M$. Since $M$ has
  positive spectrum, the corresponding sequence of generators also
  converges, to an upper-triangular matrix of RRS type that is the
  principal matrix logarithm $L$ of $M$.
  
  Clearly, $L$ is unique when no elementary Jordan block of $M$ occurs
  more than once. Moreover, it is always the only upper-triangular
  one. To see this, bring $M$ to its JNF,
  $J^{}_{M} = T \nts M \ts T^{-1}$, such that blocks with the same
  $\lambda$ are aligned as neighbours. This is possible, and $T$ can
  be chosen to be upper triangular. Assume $M = \ee^Q$ with $Q$ upper
  triangular. Since $[M,Q]=0$, we know that $Q$ respects the
  generalised eigenspaces of $M$, because $[(M-\lambda \one)^k,Q]=0$
  holds for all $\lambda \in \sigma (M)$ and $k\in\NN$.
  
  On the other hand, also $T Q \ts T^{-1}$ is still upper triangular,
  and consists of (bigger) diagonal blocks according to the
  generalised eigenspaces of $M$. Each such block has itself a block
  structure with elementary Jordan blocks for the same eigenvalue. We
  can now apply Lemma~\ref{lem:J-block} to see that we indeed have a
  unique upper-triangular real logarithm. This structure is preserved
  under transforming back to $M$ and $Q$, so $Q=L$ follows.
  
  Since any RRS matrix is upper triangular, the last claim is clear.
\end{proof}

It remains to better understand what happens for an RMS Markov matrix
$M = M(r)$ with $\sigma (M) \subset \RR_{+}$ and repeated Jordan
blocks. Invoking the approximation argument from the proof of
Theorem~\ref{thm:RMS-RRS}, let $(r_m)^{}_{m\in\NN}$ be a sequence of
probability vectors with $\lim_{m\to\infty} r_m = r$ and the property
that the spectrum of each $r_m$ is simple, which certainly
exists. Then, $\lim_{m\to\infty} M( r_m) = M$, while each $M (r_m)$
has a unique real logarithm, $L_m$ say, which is its principal
logarithm and of RRS type. Both properties are preserved in the limit,
thus hold for $L = \lim_{m\to\infty} L_m$. When all $L_m$ are Markov
generators, then so is $L$.

The remaining question is whether a non-singular $M$ can be embeddable
without being RRS embeddable. This cannot happen when no elementary
Jordan block for any of its (necessarily positive) eigenvalues is
repeated, because $L$ is then the only real logarithm of $M$. So,
consider an RMS Markov matrix $M=M(r)$ with a repeated Jordan block,
and assume it satisfies $M=\ee^Q$ for some Markov generator, which is
then one of many real logarithms of $M$. Via a small perturbation of
$r$, we see that arbitrarily close to $M$ are RMS Markov matrices
without repeated Jordan blocks. Via \cite[Prop.~4]{King} and
\cite[Thm.~7]{Davies}, there must also be embeddable ones (even with
simple spectrum), then via their principal matrix logarithms, which
are upper triangular and of RRS type. By a standard limit argument, we
then get $M = \ee^L$ and thus, via Theorem~\ref{thm:RMS-RRS}, the
following result.

\begin{coro}\label{coro:RMS-RRS}
  Let\/ $M\in\cM_d$ with\/ $\dm = \bell_n$ be an RMS Markov matrix
  with positive spectrum, not necessarily simple.  Then, the following
  properties are equivalent.
\begin{enumerate}\itemsep=2pt  
\item $M$ is embeddable.
\item $M$ is embeddable with an upper-triangular Markov generator.
\item $M$ is embeddable with a Markov generator of RRS type.
\item The principal matrix logarithm\/ $L$ of\/ $M$ is a Markov
  generator.
\end{enumerate}  
In this case, the embedding is unique in the sense that\/ $L$ is the
only upper-triangular real matrix logarithm of\/ $M$. \qed
\end{coro}

For the practical computation of the principal matrix logarithm in the
generic case, we refer back to Remark~\ref{rem:cyclic-1}, though this
will be of limited value in view of the rapid growth of
$\dm = \bell_n$ as a function of $n$.  Whenever repeated Jordan blocks
show up, there will be further real logarithms that are not of
upper-triangular form. However, in view of the particular structure of
the recombination process with the partition lattice, such cases are
of limited interest and thus not considered here.

We are now ready to embark on an investigation of recombination
matrices and their algebraic and embedding structure.

\section{Recombination for two and three sites}\label{sec:2-3}

Let us begin with the simplest recombination scheme, with two sites,
where we write $12$ and $1|2$ for the two possible partitions.  In
discrete time, the Markov transition graph is
\begin{equation}\label{eq:d-2}
\raisebox{-13pt}{\small
\begin{tikzpicture}[->,>=stealth',shorten >=1pt,auto,node
    distance=3cm, semithick]
   \tikzstyle{every state}=[fill=none,text=black]
   \node[state] (A) {$\, 12 \, $};
   \node[state] (B) [right of=A] {$1 | 2$};
   \path
   (A) edge [loop left]   (A)
   (B) edge [loop right]  (B)
   (A) edge [bend left] node [below] {$r^{\phantom{1}}_{1|2}$} (B);
\end{tikzpicture}}
\end{equation}
and has only one free parameter, $a=r^{}_{1|2} \in [0,1]$, where here
and below a loop at a node represents the remaining probability. The
corresponding Markov matrix with parameter vector $r$ (according to
Eq.~\eqref{eq:M-marg} and Definition~\ref{def:RMS}) reads
\begin{equation}\label{eq:2-sites}
    M \, = \, M(r) \, = \,
    \begin{pmatrix} 1{-} \ts a & a \\  0 & 1 \end{pmatrix} ,
\end{equation}
so $r = (1-a,a)$, which is embeddable for $0 \leqslant a < 1$ by
Fact~\ref{fact:King}. Indeed, observing that
\[
  \exp \begin{pmatrix}  - \alpha & \alpha \\ 0 & 0
  \end{pmatrix} \, = \, \begin{pmatrix} 
  \ee^{-\alpha} & 1 {-} \ts \ee^{-\alpha} \\ 0 & 1 \end{pmatrix} ,  
\]
one sees that $a = 1 - \ee^{-\alpha}$, hence $\alpha = - \log (1-a)$,
is the relation between the probability $a$ in discrete time and the
rate $\alpha$ in continuous time. The embedding in this case is unique
by Fact~\ref{fact:King}, where $M = \ee^Q$ with
\begin{equation}\label{eq:Q-2-sites}
  Q \, = \, - \log (1\nts -a) \begin{pmatrix}  - 1 & 1 \\ 0 & 0
  \end{pmatrix} \, = \, - \frac{\log (1\nts -a)}{a}
  \bigl( M \nts - \one \bigr)  .
\end{equation}
This consequence of Kendall's theorem (Fact~\ref{fact:King}) can 
be summarised as follows.

\begin{coro}\label{coro:2-sites}
  The Markov matrix\/ $M$ from~\eqref{eq:2-sites} for recombination
  at\/ $2$ sites is embeddable if and only if it is non-singular,
  which is equivalent to\/ $0\leqslant a < 1$. In this case, one has\/
  $M=\ee^Q$ with the generator from~\eqref{eq:Q-2-sites}, and the
  embedding is unique.  \qed
\end{coro}
  
\begin{figure}
\begin{tikzpicture}[->,>=stealth',shorten >=1pt,auto,node
      distance=3cm,semithick]
 \tikzstyle{every state}=[fill=none,text=black]
  \node[state] (E)                    {$\nts 1 | 2 | 3 \nts $};
  \node[state] (D) [left of=E]       {$\,3 | 1 2 $};
  \node[state] (C) [left of=D]       {$2 \ts |13$};
  \node[state] (B) [left of=C]       {$\nts 1|23 \nts$};
  \node[state] (A) [left of=B]       {$\, 123\,\ts $};  
 \path (A) edge [loop below]  (A)
  (B) edge [loop below]     (B)
  (C) edge [loop above]   (C)
  (D) edge [loop above]   (D)
  (E) edge [loop above]   (E)
  (B) edge [bend right=60]  node[above] {$r^{23}_{2|3}$}  (E)
  (C) edge [bend right=50]  node[above] {$r^{13}_{1|3}$}  (E)
  (D) edge [bend right=40] node[above] {$r^{12}_{1|2}$}  (E)
  (A) edge [bend left=60] node[below] {$r^{}_{1|2|3}$}   (E)
  (A) edge [bend left=50] node {$r^{}_{3 \ts |12}$} (D)
  (A) edge [bend left=40] node {$r^{}_{2 \ts |13}$} (C)
  (A) edge [bend left] node[below] {$r^{}_{1|23}$} (B) ;
\end{tikzpicture}
\caption{General Markov transition graph for discrete-time
  recombination with three sites; see text for
  details. \label{fig:d-3-gen} }
\end{figure}
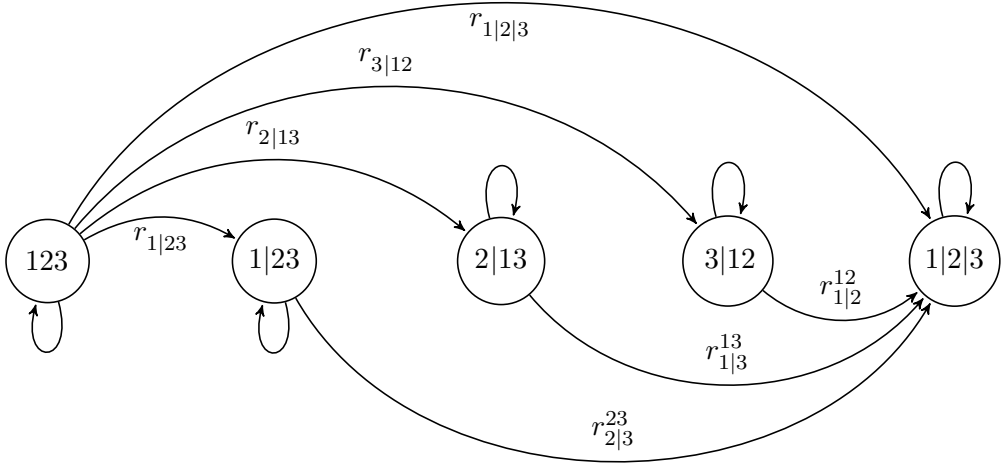

Let us next look at the still fairly transparent situation with three
sites. To simplify notation, we again write partitions with vertical
lines, so $2|13$ instead of $\bigl\{ \{ 2 \} , \{1,3\} \bigr\}$ or
$1|2|3$ instead of $\bigl\{ \{1\},\{2\},\{3\}\bigr\}$ and so on. In
discrete time, the most general transition graph is shown in
Figure~\ref{fig:d-3-gen}, where the defining transition probabilities
are $a=r^{}_{1|23}$, $b=r^{}_{2|13}$, $c=r^{}_{3|12}$ and
$d=r^{}_{1|2|3}$. The remaining ones, for consistency with
\eqref{eq:M-marg}, are the marginal probabilities
\begin{align*}
  r^{12}_{1|2} \, & = \, r^{}_{1|2|3} + r^{}_{1|23} + r^{}_{2|13} \ts , \\
  r^{23}_{2|3} \, & = \, r^{}_{1|2|3} + r^{}_{2|13} + r^{}_{3|12} \ts , \\
  r^{13}_{1|3} \, & = \, r^{}_{1|2|3} + r^{}_{1|23} + r^{}_{3|12} \ts .
\end{align*}
The most general Markov matrix in this case thus reads
\begin{equation}\label{eq:M-3-sites}
   M \, = \, \begin{pmatrix} 1{-}a{-}b{-}c{-}d & a & b & c & d \\
   0 & 1{-}b{-}c{-}d & 0 & 0 & b{+}c{+}d \\
   0 & 0 & 1{-}a{-}c{-}d & 0 & a{+}c{+}d \\
   0 & 0 & 0 & 1{-}a{-}b{-}d & a{+}b{+}d \\
   0 & 0 & 0 & 0 & 1 \end{pmatrix}
\end{equation}
with $a,b,c,d \geqslant 0$ and $a+b+c+d \leqslant 1$. Due to its 
upper-triangular structure, one has $\det (M) \geqslant 0$, with
$\det (M) > 0$ if and only if $a+b+c+d < 1$. Note that this condition
automatically forces all diagonal entries of $M$ to be strictly
positive.

When $\det (M) > 0$, the spectral radius of $A = M - \one$ is
$\varrho^{}_{A} <1$, and the principal matrix logarithm of $M$ is
given by the convergent series for $\log(M)$ from
Lemma~\ref{lem:real-log-exists}, which is a real matrix.  In the
generic case that $M$ has simple spectrum, this is the \emph{only}
real logarithm of $M$ by Fact~\ref{fact:Culver}.  The series also
gives a real logarithm in the case of degenerate eigenvalues, but
there will then be others as well, since non-trivial Jordan blocks
cannot occur in this case. The series can be computed for all
non-singular $M$ from \eqref{eq:M-3-sites}, and $\log (M)$ reads
\[
   \begin{pmatrix} 
   \log (1{-}a{-}b{-}c{-}d) & \alpha & \beta & \gamma & \delta \\
   0 & \!\log (1{-}b{-}c{-}d)\! & 0 & 0 & - \log (1{-}b{-}c{-}d) \\
   0 & 0 & \!\log (1{-}a{-}c{-}d)\! & 0 & -\log (1{-}a{-}c{-}d) \\
   0 & 0 & 0 & \!\log (1{-}a{-}b{-}d)\! & -\log (1{-}a{-}b{-}d) \\
   0 & 0 & 0 & 0 & 0 \end{pmatrix}
\]
with 
\begin{equation}\label{eq:3-parameters}
\begin{split}
 \alpha \, = \, & \log \myfrac{1{-}b{-}c{-}d}{1{-}a{-}b{-}c{-}d} \; , \quad
 \beta \, = \, \log \myfrac{1{-}a{-}c{-}d}{1{-}a{-}b{-}c{-}d} \; , \quad
 \gamma \, = \, \log \myfrac{1{-}a{-}b{-}d}{1{-}a{-}b{-}c{-}d} \; , \\[2mm]
     & \quad \text{and} \quad
     \delta \, = \, \log \myfrac{(1{-}a{-}b{-}c{-}d)^2}{(1{-}b{-}c{-}d)
     (1{-}a{-}c{-}d)(1{-}a{-}b{-}d)} \ts .
\end{split}      
\end{equation}
This can easily be checked with a computer algebra program, or
explicitly by bringing $M$ to diagonal form, with the information from
Table~\ref{tab:eigen}, then using the real logarithm on the diagonal
elements, and finally transforming back to upper-triangular form.

\begin{table}
  \caption{Eigenvalues and right eigenvectors of the matrices $R$ from
    \eqref{eq:R-3-sites}. For any eigenvalue $\mu^{}_{\!\cA}$, the
    corresponding eigenvector has a $1$ in each position that belongs
    to a partition $\cB$ with $\cB
    \succcurlyeq\cA$. \label{tab:eigen}}
\begin{center}
\begin{tabular}{|c|c|c|c|c|c|}\hline
eigenvalue & $\!\begin{array}{c}\mu^{}_{123\vphantom{|}} \\ 
       -\alpha{-}\beta{-}\gamma{-}\delta \end{array}\!$ & 
   $\begin{array}{c} \mu^{}_{1|23} \\ -\beta{-}\gamma{-}\delta \end{array}$ & 
   $\begin{array}{c} \mu^{}_{2|13} \\ -\alpha{-}\gamma{-}\delta \end{array}$ & 
   $\begin{array}{c} \mu^{}_{3|12} \\ -\alpha{-}\beta{-}\delta \end{array}$ & 
   $\;\begin{array}{c} \mu^{}_{1|2|3} \\ 0 \end{array}\;$ \\  \hline
eigenvector & $\begin{pmatrix} 1 \\ 0 \\ 0 \\ 0 \\ 0 \end{pmatrix}$ &
  $\begin{pmatrix} 1 \\ 1 \\ 0 \\ 0 \\ 0 \end{pmatrix}$ &
  $\begin{pmatrix} 1 \\ 0 \\ 1 \\ 0 \\ 0 \end{pmatrix}$ &
  $\begin{pmatrix} 1 \\ 0 \\ 0 \\ 1 \\ 0 \end{pmatrix}$ &
  $\begin{pmatrix} 1 \\ 1 \\ 1 \\ 1 \\ 1 \end{pmatrix}$ \\ \hline
\end{tabular}
\end{center}
\end{table}

Note that, due to $\det (M) > 0$ together with $a,b,c \geqslant 0$,
the real numbers $\alpha, \beta, \gamma$ in \eqref{eq:3-parameters}
are automatically non-negative, while $\delta$ is well defined but can
be negative. A simple calculation shows that $R=\log (M)$ is a real
matrix of the form
\begin{equation}\label{eq:R-3-sites}
    R \, = \begin{pmatrix} * & \alpha & \beta & \gamma & \delta \\
    0 & * & 0 & 0 & \beta {+} \gamma {+} \delta \\
    0 & 0 & * & 0 & \alpha {+} \gamma {+} \delta \\
    0 & 0 & 0 & * & \alpha {+} \beta {+} \delta \\
    0 & 0 & 0 & 0 & 0 \end{pmatrix} ,
\end{equation}
where the $*$ in each row is the unique real number that enforces row
sum $0$. Note that $R$ thus satisfies the marginalisation relations of
Eq.~\eqref{eq:Q-marg}. Here, as follows from \eqref{eq:3-parameters},
the parameters $\alpha,\beta,\gamma,\delta$ are simple linear
combinations of the eigenvalues of $R$. What is more, \emph{all}
matrices of this form share the same parameter-independent set of
eigenvectors given in Table~\ref{tab:eigen}. This shows that all
matrices of the form \eqref{eq:R-3-sites} are simultaneously
diagonalisable, and hence also commute with one another.  We can now
state the following result.

\begin{prop}\label{prop:3-sites}
  Let\/ $M$ be a Markov matrix of the form~\eqref{eq:M-3-sites}, hence
  with\/ $a,b,c,d\geqslant 0$ and\/ $a+b+c+d \leqslant 1$. If\/
  $\det (M) > 0$, which happens if and only if\/ $a+b+c+d < 1$, the
  following properties are equivalent.
 \begin{enumerate}\itemsep=2pt
 \item[(C1)] $M$ is embeddable.
 \item[(C2)] $M$ is embeddable with a Markov generator of the 
                   form~\eqref{eq:R-3-sites}, which is the principal
                   matrix logarithm\/ $L$ of\/ $M$.
 \item[(C3)] $(1-a-b-c-d)^2 \geqslant (1-b-c-d)(1-a-c-d)(1-a-b-d)$.
 \end{enumerate}  
 When\/ $M$ is embeddable and has simple spectrum, the embedding is
 unique. If the spectrum is degenerate, $M$ is still
 diagonalisable. Here, an embedding is never unique, but no generator
 other than\/ $L$ can be of the form\/ \eqref{eq:R-3-sites} or
 otherwise upper triangular.
 \end{prop}

\begin{proof}
  First, let $M$ be non-singular with simple spectrum. Then, it has a
  unique real matrix logarithm by Theorem~\ref{thm:real-log-and-gen},
  which must then be $R$ from \eqref{eq:R-3-sites} as a result of the
  above calculations, and (C1) $\Leftrightarrow$ (C2) is then
  clear. In fact, $R$ is the principal matrix logarithm of $M$, so
  $R=L$, in line with our general result in
  Corollary~\ref{coro:RMS-RRS}.

  To show (C2) $\Rightarrow$ (C3), observe that the embedding implies
  $\alpha, \beta, \gamma, \delta\geqslant 0$.  Since
  $\alpha, \beta, \gamma \geqslant 0$ holds automatically, the only
  extra condition is $\delta\geqslant 0$, which is equivalent with
  (C3) by \eqref{eq:3-parameters}.

  Finally, to establish (C3) $\Rightarrow$ (C2), we note that (C3)
  implies $\delta\geqslant 0$. On the other hand, $a,b,c \geqslant 0$
  imply $\alpha, \beta, \gamma \geqslant 0$, from which the
  non-negativity of all off-diagonal entries of $R$ follows. As all
  row sums are $0$, the matrix $R$ is indeed a Markov generator.
  
  Now, again in line with our various closure arguments in and after
  the proof of Theorem~\ref{thm:RMS-RRS}, it is clear that the
  equivalence of the conditions remains true for degenerate spectra.
  The continuity arguments can be repeated here under the simpler
  setting of diagonalisable matrices, the latter being approximated by
  embeddable matrices with simple spectra.

  Uniqueness in the case of simple spectrum is a consequence of
  Fact~\ref{fact:Culver}, while non-uniqueness in the presence of
  repeated eigenvalues is also discussed in \cite{Culver}, where the
  diagonalisability of $M$ follows via the set of right eigenvectors
  from Table~\ref{tab:eigen}.  Now, assuming $\ee^R = \ee^{R'}$ with
  two matrices of the form \eqref{eq:R-3-sites} implies that $R$ and
  $R'$, which then commute, have the same eigenvalues. Thus, they 
  have the same parameters, and hence satisfy $R=R'$.
\end{proof}

There is a bit more to say on the family of matrices of the form
\eqref{eq:R-3-sites}. Clearly, one has
$R = \alpha X_1 + \beta X_2 + \gamma X_3 + \delta \ts Y$ with
\begin{align}
\allowdisplaybreaks
   \label{eq:N1}  
   X_1 \, & = \, \begin{pmatrix} -1 & 1 & 0 & 0 & 0 \\ 0 & 0 & 0 & 0 & 0 \\
     0 & 0 & -1 & 0 & 1 \\ 0 & 0 & 0 & -1 & 1 \\ 0 & 0 & 0 & 0 & 0
     \end{pmatrix}  , \quad     
   X_2 \, = \, \begin{pmatrix} -1 & 0 & 1 & 0 & 0 \\ 0 & -1 & 0 & 0 & 1 \\
     0 & 0 & 0 & 0 & 0 \\ 0 & 0 & 0 & -1 & 1 \\ 0 & 0 & 0 & 0 & 0
     \end{pmatrix}  , \\[2mm] 
   \label{eq:N2}
   X_3 \, & = \, \begin{pmatrix} -1 & 0 & 0 & 1 & 0 \\ 0 & -1 & 0 & 0 & 1 \\
     0 & 0 & -1 & 0 & 1 \\ 0 & 0 & 0 & 0 & 0 \\ 0 & 0 & 0 & 0 & 0
     \end{pmatrix}  , \quad
   Y \, = \, \begin{pmatrix} -1 & 0 & 0 & 0 & 1 \\ 0 & -1 & 0 & 0 & 1 \\
     0 & 0 & -1 & 0 & 1 \\ 0 & 0 & 0 & -1 & 1 \\ 0 & 0 & 0 & 0 & 0
     \end{pmatrix} .   
\end{align}
These four matrices mutually commute, and satisfy the relations
\[
\begin{split}
  X^{2}_{i} \, & = \, - X^{}_{i} \; , \quad Y^2 \, = \, - Y
  \quad\text{together with}\\
  X_i X_j  \, & = Y - X_i - X_j  \; ,
  \quad X_i \ts Y  \, = \, - X_i  \, .
\end{split}
\]
This shows that they span a four-dimensional Abelian matrix algebra,
which is a particularly nice and simple algebraic structure. However,
as we shall see shortly, this structure does not generalise to more
than three sites.  \medskip

Let us now look at the important special case of a single-crossover
process for three sites in discrete time. Here, each part of the
current partition can experience a split into at most two contiguous
blocks in one time step, which means that only interval partitions
emerge. In particular, $13|2$ can never be reached, and $1|2|3$ not in
a single step from $123$. For $3$ sites, this process agrees with the
multiple coupon collection process, for which the embedding problem
was solved in \cite{mccp}; the process is different for more than $3$
sites.

In our present setting, we have to consider the matrix $M$ from
\eqref{eq:M-3-sites} with $b=d=0$, which implies that we are in the
situation of Figure~\ref{fig:sc-d-3}.  However, this means that we get
\[
    \delta \, = \, \log \myfrac{(1-a-c)^2}{(1-a)(1-a-c)(1-c)} \, = \,
    \log \myfrac{1-a-c}{1-a-c+ac} \, \leqslant \, 0
\]
which can only be $0$ if $a=0$ or $c=0$, which effectively brings us
back to two sites. So, the standard non-trivial Markov matrix for
single-crossover recombination in discrete time, for $3$ sites, is
\emph{never} embeddable when $ac>0$. While single-crossover
recombination is a much-studied model, see \cite{WBB,BvW14,Martinez}
and references therein, it is not compatible with an underlying
continuous-time process of any type; see Remark~\ref{rem:g-embed} 
below.

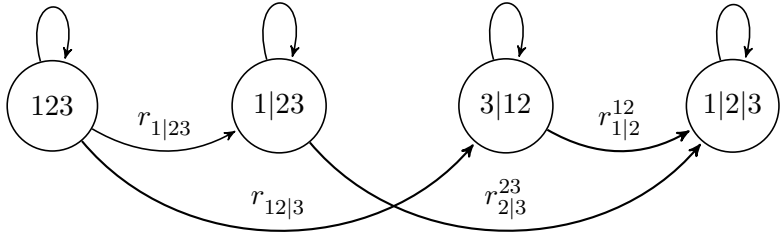
\begin{figure}
\begin{tikzpicture}[->,>=stealth',shorten >=1pt,auto,node
    distance=3.0cm,semithick]
  \tikzstyle{every state}=[fill=none,text=black]
  
  \node[state] (D)                      {$1 | 2 | 3 $};
  \node[state] (C) [left of=D]     {$\, 3 | 1 2 \,$};
  \node[state] (B) [left of=C]     {$\,1 | 2 3\,$};
  \node[state] (A) [left of=B]      {$\ts \; 1 2 3 \; \ts$};

  \path
  (A) edge [loop above]  (A)
  (B) edge [loop above]  (B)
  (C) edge [loop above]  (C)
  (D) edge [loop above]  (D)

  (A) edge [bend right]   node {$r^{}_{1|23}$}  (B);
   \draw[->, thick] (C) to[bend right=30] node {$r^{12}_{1|2}$} (D);
   \draw[->, thick] (B) to[bend right=50] node {$r^{23}_{2|3}$} (D);
   \draw[->, thick] (A) to[bend right=50] node
         {$r^{\phantom{12}}_{12|3}$}  (C);
\end{tikzpicture}
\caption{Single-crossover recombination for three sites in discrete
  time. \label{fig:sc-d-3} }
\end{figure}

Indeed, this can also be seen independently as follows. Working with
the four states of Figure~\ref{fig:sc-d-3}, the most general generator
would be
\[
  Q \, = \, \begin{pmatrix} -\alpha{-}\beta & \alpha & \beta & 0 \\
      0 & -\beta & 0 & \beta \\  0 & 0 & - \alpha & \alpha \\
       0 & 0 & 0 & 0   \end{pmatrix} \, = \, Q_{\alpha} + Q_{\beta}
\]
with the rates $\alpha=\rho^{}_{1|23}$ and $\beta=\rho^{}_{3|12}$ in
obvious partition notation, where the row and column labels follow the
order of the nodes in Figure~\ref{fig:sc-d-3}. As $Q_{\alpha}$ and
$Q_{\beta}$ commute, we get
\[
\begin{split}  
  \exp & (Q_{\alpha} + Q_{\beta}) \, = \,
  \exp (Q_{\alpha}) \ts \exp (Q_{\beta}) \\[2mm]
  & = \, \begin{pmatrix} 
     \ee^{-\alpha-\beta} &  \ee^{-\beta}(1-\ee^{-\alpha}) &
      \ee^{-\alpha}(1-\ee^{-\beta}) & (1-\ee^{-\alpha})(1-\ee^{-\beta}) \\
    0 & \ee^{-\beta} & 0 &  1-\ee^{-\beta} \\
    0 & 0 & \ee^{-\alpha} & 1 - \ee^{-\alpha} \\
    0 & 0 & 0 & 1   \end{pmatrix},
\end{split}
\]
as follows from an elementary calculation. When interpreted in terms
of the graph from Figure~\ref{fig:d-3-gen}, we obtain
\begin{align*}
  r^{}_{1|23} \, & = \, \ee^{-\beta} (1-\ee^{-\alpha}) \ts  , \;
  r^{}_{2|13} \, = \, 0 \ts , \;
  r^{}_{3|12} \,  = \, \ee^{-\alpha} (1-\ee^{-\beta}) \ts , \\[1mm]
  r^{}_{1|2|3} \, & = \,  (1-\ee^{-\alpha})(1-\ee^{-\beta}) \\[1mm]
  r^{12}_{1|2} \, & = \, 1-\ee^{-\alpha} , \;
      r^{23}_{2|3} \, = \, 1-\ee^{-\beta} , \;
      r^{13}_{1|3} \, = \, 1-\ee^{-(\alpha+\beta)} \ts .
\end{align*}
For the full Markov semigroup, one simply replaces $\alpha$ by
$t\alpha$ and $\beta$ by $t\beta$ in the above expressions, which also
allows to look at the asymptotic behaviour for $t\to\infty$.

\begin{remark}\label{rem:g-embed}
  More generally, one can ask whether a potentially non-embeddable
  Markov matrix $M$ of the form \eqref{eq:M-3-sites} can be a product
  of embeddable ones, that is, whether
\[
     M \, = \, \ee^{Q_1} \nts \cdots  \ee^{Q_m}
\]
with recombination generators $Q_1, \ldots , Q_m$ is still possible.
Since the latter commute with one another, this would imply
$M = \exp (Q_1 + \ldots + Q_m)$, where $Q' = \sum_{i=1}^{m} Q_i$ is
again a recombination generator, because the relations in
\eqref{eq:Q-marg} are linear. Consequently, one would get
$M = \ee^{Q'}$, and hence standard embeddability.  By a result
of Johansen \cite{Joh73}, this also excludes embeddability into a
time-inhomogenous process; see also \cite[Sec.~6]{BS3}. 
\exend
\end{remark}

\begin{remark}
The incompatibility of single-crossover recombination with standard 
embeddability in a time-homogeneous process  is also clear from the
transitivity property of embeddable Markov matrices; see
\cite[Prop.~2.1]{BS1} and references given there for background.  

In fact, the simple condition on transitivity consistency leads to a
more general observation as follows.  Consider discrete-time
recombination with\/ $n$ sites, and take the lattice
$\cP'\subseteq \cP (S)$ of partitions that is generated by all
partitions $\cA$ with\/ $r^{}_{\! \cA} > 0$. Then, if $r^{}_{\cB}=0$
for any $\cB\in\cP'$, the process cannot be embeddable.  \exend
\end{remark}

Unfortunately, Remark~\ref{rem:g-embed} does no longer apply when the
algebra generated by the RRS matrices is non-commutative. Let us next
analyse the simplest case where this happens.

\section{Recombination for four sites}\label{sec:4}

To understand why $2$ and $3$ sites are special, and why their
treatment does not give the right idea of the general structure, we
need to analyse $4$ sites in some detail. This is the smallest number
of sites where a non-linear parameter dependence of (some) elements
and eigenvalues of $M$ emerges, which will lead to a rather different
algebraic structure. Here, we have $S = \{1,2,3,4 \}$, hence
$\dm =B^{}_{4} = 15$, and thus $14$ free parameters, as coded by the
probability vector $r = (r^{}_{\!\cA})^{}_{\cA\in\cP(S)}$. The $15$
eigenvalues of $M$ are the $\lambda^{}_{\cA}$ with $\cA\in\cP (S)$,
with $\lambda^{}_{1234} = r^{}_{1234}$ and $\lambda^{}_{1|2|3|4} = 1$,
while the remaining ones, using Eqs.~\eqref{eq:marg-prob} and
\eqref{eq:eigen-M}, are given by
\[
\begin{split}
   \lambda^{}_{i | jk\ell} \, & = \, r^{\, j k\ell}_{ \{ j k \ell \} }
       \, = \, r^{}_{1234 \vphantom{|}} + r^{}_{i | j k \ell} \ts , \\[1mm]
   \lambda^{}_{i | j | k \ell } \, & = \, r^{\, k \ell}_{ \{ k \ell \} } \, = \,
        r^{}_{1234 \vphantom{|}} +  r^{}_{i | j k \ell} + r^{}_{j | i k \ell}
        + r^{}_{ij | k \ell } + r^{}_{i | j | k \ell } \ts , \\[1mm]
  \lambda^{}_{ij | k \ell} \, & = \,  r^{\, i j }_{ \{ i j \} }
             r^{\, k \ell}_{ \{ k \ell \} }  \, = \,
  \lambda^{}_{k | \ell | i j} \cdot \lambda^{}_{i | j | k \ell } \ts ,
\end{split}
\]
where the first, second and third line account for $4$, $6$ and $3$
eigenvalues, respectively. In all cases, the four indices represent a
permutation of the elements of $S$, which is a widely used shorthand
in this setting. The mentioned non-linearity is clear from the third
line.

\subsection{Generator structure}
Let us look at a Markov generator, whose first row is arranged as
indicated in Table~\ref{tab:4-sites}, with $\rho (1234)$ again
following from the condition that the row sum is $0$. When
$M = \ee^Q$, via the diagonal elements, we get the following relations
for the eigenvalues,
\begin{equation}\label{eq:4-eigen}
\begin{split}
  \mu^{}_{1234} \, & = \, \log (\lambda^{}_{1234} ) \, = \,
    - \! \sum_{\pmax^{\nts S} \ne \cA \in \cP(S)} \! \rho (\cA ) 
     \, = \, \rho (\pmax) \ts ,  \\[1mm]
  \mu^{}_{i | jk\ell} \, & = \, \log (\lambda^{}_{i | jk\ell} ) \, = \,
    - \! \sum_{\pmax^{\nts jk\ell} \ne \cA \in \cP (\{j,k,\ell\})} 
            \! \rho^{\ts jk\ell}_{\cA} , \\[1mm]
  \mu^{}_{ij | k\ell} \, & = \, \log (\lambda^{}_{ij | k\ell}) \, = \,
    - \bigl( \rho^{\ts ij}_{i | j} + \rho^{\ts k\ell}_{k | \ell} \bigr) , \\[1mm]
  \mu^{}_{i | j | k\ell} \, & = \, \log (\lambda^{}_{i | j | k\ell} ) \, = \, 
    - \rho^{\ts k\ell}_{k | \ell} \ts , \\[1mm]      
   \mu^{}_{1|2|3|4} \, & = \, \log (\lambda^{}_{1|2|3|4} ) \, = \, 0 \ts ,  
\end{split}
\end{equation}
with the same shorthand as above. The eigenvalue expressions also
reveal that we have three non-trivial linear relations among the
eigenvalues of $Q$, namely
\[
    \mu^{}_{ij | k\ell} \, = \mu^{}_{i | j | k\ell} + \mu^{}_{k | \ell | ij} \ts ,
\]
which is the reason why we cannot determine the $14$ parameters from
the eigenvalues of $Q$ by solving a system of \emph{linear} equations.
This problem did not occur in Section~\ref{sec:2-3}.

Another difference emerges as follows. In the generic case of simple
spectrum, the generator $Q$ is diagonalisable. When degeneracies
occur, this can still be true, but it need not, so non-trivial Jordan
blocks are possible, as is known from the analysis of the special case
of interval partitions \cite{interval}. To be more specific, let
$\cA^{\trans}$ denote the unit column vector with a $1$ in position
$\cA$ and $0$ everywhere else. Then, for any eigenvalue
$\mu^{}_{\! \cA}$ where $\cA\in\cP(S)$ has at most one non-singleton
part, the corresponding right eigenvector $v^{\trans}_{\!\cA}$ is
parameter independent and reads
\[
    v^{\trans}_{\!\cA} \, = \sum_{\cB \succcurlyeq \cA} \cB^{\trans} ,
\]
which is the structure we saw in Table~\ref{tab:eigen} for $3$ sites
as well.  For the remaining partitions,
$\cA \in \{ 12|34, 13|24, 14|23 \}$, we find
\begin{align*}
   v^{\trans}_{12|34} \, & = \, (12|34)^{\trans} -
              \myfrac{b^{}_{1}}{-b^{}_{1} + b^{}_{2} + b^{}_{3} 
              + c^{}_{2} + c^{}_{3} + c^{}_{4} + c^{}_{5} + d^{}_{1}} 
              (1234)^{\trans}  , \\
   v^{\trans}_{13|24} \, & = \, (13|24)^{\trans} -
              \myfrac{b^{}_{2}}{\, b^{}_{1} - b^{}_{2} + b^{}_{3} 
              + c^{}_{1} + c^{}_{3} + c^{}_{4} + c^{}_{6} + d^{}_{1}} 
              (1234)^{\trans}  , \\
   v^{\trans}_{14|23} \, & = \, (14|23)^{\trans} - 
              \myfrac{b^{}_{3}}{\, b^{}_{1} + b^{}_{2} - b^{}_{3} 
              + c^{}_{1} + c^{}_{2} + c^{}_{5} + c^{}_{6} + d^{}_{1}} 
              (1234)^{\trans} ,
\end{align*}
provided the denominators are non-zero. In this case, the matrix $Q$
is still diagonalisable. In the excluded cases, we get a non-trivial
Jordan block, as in \cite{interval}.

\begin{table}
\caption{Transitions from $\pmax$ to $\cA\ne\pmax$   \label{tab:4-sites} }
\begin{center}
\begin{tabular}{|c|c||c|c|} \hline
$\cA$ & $\rho^{}_{\! \cA}$ & $\cA$ & $\rho^{}_{\! \cA}$ \\ \hline
$1|234$ & $a^{}_{1}$ & $1|2|34$ & $c^{}_{1}$ \\
$2|134$ & $a^{}_{2}$ & $1|3|24$ & $c^{}_{2}$ \\
$3|124$ & $a^{}_{3}$ & $1|4|23$ & $c^{}_{3}$ \\
$4|123$ & $a^{}_{4}$ & $2|3|14$ & $c^{}_{4}$ \\
$12|34$ & $b^{}_{1}$ & $2|4|13$ & $c^{}_{5}$ \\
$13|24$ & $b^{}_{2}$ & $3|4|12$ & $c^{}_{6}$ \\
$14|23$ & $b^{}_{3}$ & $1|2|3|4$ & $d^{}_{1}$ \\ \hline
\end{tabular}
\end{center}
\end{table}

The remaining potentially non-zero transition rates follow from 
marginalisation, and are
\[
  i | jk\ell \rightarrow i | j | k\ell :
     \rho^{\ts jkl}_{j | k\ell} \ts , \quad
  i | jk\ell \rightarrow i | j | k | \ell :
     \rho^{\ts jk\ell}_{j | k | \ell} \ts , \quad
  ij | k\ell \rightarrow i | j | k\ell : \rho^{\ts ij}_{i | j}  \ts ,
\]
while all other transition rates vanish, which applies in particular
to $12|34\rightarrow 1|2|3|4$, because never more than one part can be
refined in one step.

The marginal rates in terms of the parameters from
Table~\ref{tab:4-sites} read
\begin{align*}
    \rho^{\ts 234}_{2|34} \nts & = a^{}_{2} + b^{}_{1} + c^{}_{1} \ts , & 
    \rho^{\ts 134}_{1|34} \nts & = b^{}_{1} + c^{}_{1} + a^{}_{1} \ts , &
    \rho^{\ts 234}_{3|24} \nts & = a^{}_{3} + b^{}_{2} + c^{}_{2} \ts , &  
    \rho^{\ts 134}_{3|14} \nts & = b^{}_{3} + c^{}_{4} + a^{}_{3} \ts , \\
    \rho^{\ts 234}_{4|23} \nts & = a^{}_{4} + b^{}_{3} + c^{}_{3} \ts ,& 
    \rho^{\ts 134}_{4|13} \nts &= b^{}_{2} + c^{}_{5} + a^{}_{4} \ts , &
    \rho^{\ts 124}_{1|24} \nts & = a^{}_{1} + b^{}_{2} + c^{}_{2} \ts ,&
    \rho^{\ts 123}_{1|23} \nts & = b^{}_{3} + c^{}_{3} + a^{}_{1} \ts , \\
    \rho^{\ts 124}_{2|14} \nts & = a^{}_{2} + b^{}_{3} + c^{}_{4} \ts, & 
    \rho^{\ts 123}_{2|13} \nts & = b^{}_{2} + c^{}_{5} + a^{}_{2} \ts ,&
    \rho^{\ts 124}_{4|12} \nts & = a^{}_{4} + b^{}_{1} + c^{}_{6} \ts ,& 
    \rho^{\ts 123}_{3|12} \nts & = b^{}_{1} + c^{}_{6} + a^{}_{3} \ts ,
\end{align*}
together with
\begin{align*}             
  \rho^{\ts 234}_{2|3|4} \nts & = c^{}_{4} + c^{}_{5} + c^{}_{6} + d^{}_{1} \ts , & 
  \rho^{\ts 134}_{1|3|4} \nts & = c^{}_{2} + c^{}_{3} + c^{}_{6} + d^{}_{1} \ts , \\
  \rho^{\ts 124}_{1|2|4} \nts & = c^{}_{1} + c^{}_{3} + c^{}_{5} + d^{}_{1} \ts , & 
  \rho^{\ts 123}_{1|2|3} \nts & = c^{}_{1} + c^{}_{2} + c^{}_{4} + d^{}_{1} \ts ,
\end{align*}
and
\begin{align*}
 \rho^{\ts 12}_{1|2} \nts & = a^{}_{1} + a^{}_{2} + b^{}_{2} + b^{}_{3}
     + c^{}_{1} + c^{}_{2} + c^{}_{3} + c^{}_{4} + c^{}_{5} + d^{}_{1} \ts , \\
 \rho^{\ts 13}_{1|3} \nts & =  a^{}_{1} + a^{}_{3} + b^{}_{1} + b^{}_{3} 
     + c^{}_{1} + c^{}_{2} + c^{}_{3} + c^{}_{4} + c^{}_{6} + d^{}_{1} \ts , \\     
 \rho^{\ts 14}_{1|4} \nts & =  a^{}_{1} + a^{}_{4} + b^{}_{1} + b^{}_{2}  
     + c^{}_{1} + c^{}_{2} + c^{}_{3} + c^{}_{5} + c^{}_{6} + d^{}_{1} \ts , \\  
 \rho^{\ts 23}_{2|3} \nts & =  a^{}_{2} + a^{}_{3} + b^{}_{1}  + b^{}_{2} 
     + c^{}_{1} + c^{}_{2} + c^{}_{4} + c^{}_{5} + c^{}_{6} + d^{}_{1} \ts , \\  
 \rho^{\ts 24}_{2|4} \nts & =  a^{}_{2} + a^{}_{4} + b^{}_{1} + b^{}_{3} 
     + c^{}_{1} + c^{}_{3} + c^{}_{4} + c^{}_{5} + c^{}_{6} + d^{}_{1} \ts , \\  
 \rho^{\ts 34}_{3|4} \nts & =  a^{}_{3} + a^{}_{4} + b^{}_{2} + b^{}_{3}
     + c^{}_{2} + c^{}_{3} + c^{}_{4} + c^{}_{5} + c^{}_{6} + d^{}_{1} \ts .
\end{align*}
All the other ones vanish.

Now, in analogy to Eqs.~\eqref{eq:N1} and \eqref{eq:N2}, let
$X^{}_{\!  \cA}$ be the generator that is obtained by setting
$\rho^{}_{\! \cA} = 1$ and all other parameters to $0$, giving us $14$
generators.  Inspecting $X^{}_{12 | 34}$, one realises that its square
cannot be written as a linear combination of the $X^{}_{\! \cA}$, thus
showing that our $14$ generators do \emph{not} span a matrix
algebra. Let us thus look at the \emph{commutator} (or Lie bracket) in
our matrix setting,
\[
     [ X, Y ] \, \defeq \, X \ts Y - Y X , 
\]
which defines a bilinear product. If we know the outcome for the $196$
commutators of our $14$ generators, we then get the extension to the
real span of them via
\begin{equation}\label{eq:lin-extend}
    \Bigl[  \sum_{\cA\in\cP(S)} \! \rho^{}_{\!\cA} X^{}_{\!\cA} ,
      \sum_{\cB\in\cP (S)} \! \eta^{}_{\cB} X^{}_{\cB} \Bigr] \, = \!
      \sum_{\cA, \cB \in \cP (S)}\!  \rho^{}_{\!\cA} \ts \eta^{}_{\cB}
      \, [ X^{}_{\!\cA} , X^{}_{\cB} ] \ts .
\end{equation}
Clearly, $14$ commutators are trivial, due to $[X,X]=0$.  This leaves
us with $91$ to determine, since $[X,Y]=-[Y,X]$. Of these, precisely
$18$ are non-zero, namely
\begin{equation}\label{eq:Lie-4}
\begin{split}
  [ X^{}_{1|2|3|4} , X^{}_{ij | k\ell} ] \, & = \,  X^{}_{ij | k\ell} 
    - X^{}_{i | j | k\ell} - X^{}_{k | \ell | ij} + X^{}_{1|2|3|4}  \ts , \\[1mm]
  [ X^{}_{ij | k\ell} , X^{}_{ik | j\ell} ] \, & = \,
    X^{}_{ik | j\ell} - X^{}_{i | k | j\ell} - X^{}_{j | \ell | ik} 
    - X^{}_{ij | k\ell} + X^{}_{i | j | k\ell} + X^{}_{k | \ell | ij} \ts , \\[1mm]
  [ X^{}_{i | j | k\ell} , X^{}_{ik | j\ell} ] \, & = \, X^{}_{ik | j\ell} 
    - X^{}_{i | k | j\ell} - X^{}_{j | \ell | ik} + X^{}_{1|2|3|4} \ts ,          
\end{split}    
\end{equation}
where the first two formulas account for $3$ relations each, and the
last for $12$, so $18$ in total. Each of these non-trivial commutators
is a linear combination in the $X^{}_{\!\cA}$. Since the $14$
generators are linearly independent over $\RR$, as follows from the
structure of their first row, we have shown the following result.

\begin{lemma}\label{lem:Lie-4}
  The\/ $14$ matrices\/ $X^{}_{\!\cA}$ with\/
  $\pmax \ne \cA \in \cP (S)$ span a real matrix Lie algebra of
  dimension\/ $14$, with the non-trivial commutators being given by
  \eqref{eq:Lie-4}.  \qed
\end{lemma}

This is the new algebraic structures announced earlier, which starts
at $4$ sites and will be analysed in more generality in
Section~\ref{sec:general}.

\subsection{Embedding}
To approach the embedding question, we start from a non-singular
Markov matrix $M$ with simple spectrum, which has a unique real
logarithm by Lemma~\ref{lem:real-log-exists}, say $Q$ with $M=\ee^Q$,
where we now need to assess when $Q$ really is a Markov generator. As
it must have the marginalisation structure from \eqref{eq:Q-marg}, we
need to compute its first row only, and check when we get
$Q^{}_{\pmax \cA} = \rho^{}_{\! \cA} \geqslant 0$ for all
$\pmax \ne \cA \in \cP (S)$. With the parameters
$r^{}_{\!\cA}\geqslant 0$ from $M$ and the eigenvalue relations from
\eqref{eq:4-eigen}, one finds
\[
    Q^{}_{\pmax \pmax} \, = \, \rho^{}_{1234} \, = 
       \log (\lambda^{}_{1234}) \, \leqslant \, 0 \ts ,
\]
which ensures that the row sum is $0$, together with the following
identities.  First, one has
\[
  \rho^{}_{1 | 234} \, = \, \log \myfrac{\lambda_{1 | 234}}{\lambda^{}_{1234}}
  \, = \, \log \Bigl( 1 + \myfrac{r^{}_{1 | 234}}{r^{}_{1234}} \Bigr)
  \, \geqslant \, 0 \ts , 
\]
and analogously for the other three partitions of this kind. Here,
non-negativity of the $4$ parameters $\rho^{}_{i | jk\ell}$ is
automatic. Next, one obtains 
\[
  \rho^{}_{ij | k\ell} \, = \, r^{}_{ij | k\ell} \,
  \myfrac{\log (\lambda^{}_{1234}) - \log (\lambda_{ij | k\ell})}
    {\lambda^{}_{1234} - \lambda^{}_{ij | k\ell}} \, \geqslant \, 0 \ts ,
\]
which accounts for $3$ parameters, where the fraction is indeed always
non-negative. This condition follows form an explicit computation,
which we skip here because we present a general approach in
Section~\ref{sec:gen-emb-cond}.

Now, we come to the parameters that need not always be non-negative.
Here, in a similar fashion, we get $6$ relations, namely 
\begin{equation}\label{eq:4-c1}
  \rho^{}_{i | j | k  \ell} \,  = \, \rho^{\ts ik \ell}_{i | k \ell }
        - \rho^{}_{i | jk\ell} - \rho^{}_{ij | k\ell}  
     = \, \log \myfrac{\lambda_{i | j | k \ell} \, \lambda^{}_{1234} }
             {\lambda_{j | ik \ell} \, \lambda_{i | jk\ell} }  
            -  r^{}_{ij | k\ell} \, 
            \myfrac{\mu^{}_{1234} - \mu_{ij | k\ell}}
                 {\lambda^{}_{1234} - \lambda_{ij | k\ell}} \ts .
\end{equation}
By another computation of the same kind, we arrive at the final
identity,
\begin{equation}\label{eq:4-c2}
\begin{split}
  \rho^{}_{1|2|3|4} \, & = \, - 3 \ts \mu^{}_{1234} - 2 ( \mu_{1|234}
        + \mu_{2|134} + \mu_{3|124} + \mu_{4|123}) \\[1mm]
  & \quad \; - (\mu_{1|2|34} + \mu_{1|3|24} + \mu_{1|4|23} 
             + \mu_{2|3|14} + \mu_{2|4|13} + \mu_{3|4|12}) \\[1mm]
  & \quad \; - r^{}_{12|34} \myfrac{\mu^{}_{1234} - \mu_{12|34}}
             {\lambda^{}_{1234}-\lambda_{12|34}}    
             - r^{}_{13|24} \myfrac{\mu^{}_{1234} - \mu_{13|24}}
             {\lambda^{}_{1234}-\lambda_{13|24}}   
             - r^{}_{14|23} \myfrac{\mu^{}_{1234} - \mu_{14|23}}
             {\lambda^{}_{1234}-\lambda_{14|23}}  \ts .                  
\end{split}
\end{equation}
Here, lines one and two of the right-hand side together are always
strictly positive, while the third line is non-positive, so we get
a real condition from this equation. Later, we shall derive a
systematic method to compute these conditions, which is based on a
triangular recursion.

To also cover the case of degenerate spectra, we first recall from
\cite[Prop.~3]{King} that the set of embeddable Markov matrices is
relatively closed within the set of all Markov matrices with positive
determinant, which means that we can use a suitable continuity
argument with our conditions. The only obstacle for this are
degeneracies between eigenvalues that occur as a difference in a
denominator. For this, observe that the $2$-variable function defined
by $(x,y) \mapsto \frac{\log(x) - \log(y)}{x-y}$ for positive $x,y$
with $x \ne y$ has the unique continuous extension to $\frac{1}{x}$
for $x=y>0$ by de l'Hospital's rule; see the closely related case
in Lemma~8 of \cite[Appendix]{BBS}. This means that, for Markov
matrices with degenerate positive spectrum, we have to extend the
conditions $\rho^{}_{ij | k | \ell} \geqslant 0$ and
$\rho^{}_{1|2|3|4} \geqslant 0$ in this way, thus getting the correct
criterion for embeddability in general. We thus have the following
result.

\begin{theorem}
  A non-singular Markov matrix\/ $M$ with simple spectrum for 
  recombination with\/ $4$ sites is embeddable if and only if 
  its unique real matrix logarithm satisfies\/
  $\rho^{}_{ij | k | \ell} \geqslant 0$ and\/
  $\rho^{}_{1|2|3|4} \geqslant 0$, which are\/ $7$ conditions with the
  parameters from \eqref{eq:4-c1} and \eqref{eq:4-c2}.

  Further, when\/ $M$ is non-singular but has degeneracies in its
  spectrum, its principal matrix logarithm is still an RRS matrix and
  thus upper triangular.  The latter is a rate matrix if and only if
  the conditions from \eqref{eq:4-c1} and \eqref{eq:4-c2} are
  satisfied, where fractions of the form\/
  $\frac{\log(x) - \log(y)}{x-y}$ have to be replaced by\/
  $\frac{1}{x}$ whenever\/ $x=y$.  \qed
\end{theorem}

It is now time to move on to the general case, where we will
employ and profit from some more algebraic tools, in particular
Lie-theoretic ones.

\section{The general case}\label{sec:general}

Let now $S = \{ 1, 2, \ldots , n\}$ with $n\in\NN$ be arbitrary, but
fixed. There are $\bell_n - 1$ elementary recombination generators,
denoted by $X^{}_{\cB}$ with $\pmax \ne \cB \in \cP (S)$. For our
Markov matrices, we use a basis of $\RR^d$ that is labelled by the
partitions $\cA \in \cP (S)$. Therefore, we now adopt the slight (but
common) abuse of notation to identify the partition labels also with
the corresponding unit row vectors. Then, $\cA \ts X_{\cB}$ is a
well-defined row vector again, which is a sum that balances the
`input' and `output' for all one-part refinements of $\cA$ by $\cB$,
thus giving
\[
     \cA \ts X_{\cB} \, = \sum_{A\in \cA}
     \bigl( (\cA \setminus A ) \sqcup \cB|_{A} - \cA \bigr)
     \, = \sum_{A\in\cA} \bigl( (\cA\setminus A) \sqcup
          (\cB|_{A} - \{ A \} ) \bigr) .   
\]
With this, we would also get $\cA \ts X_{\pmax} = 0$ for all $\cA$,
hence $X_{\pmax} = 0$, which is the reason why we only need to
consider $\cB \ne \pmax$ for the generating operators.

When $n\geqslant 4$, the $X_{\cB}$ can never span a matrix algebra, as
we saw in the previous section, the crucial observation being that the
emergence of partitions with more than one non-singleton part implies
the square of the corresponding generator to have non-zero elements
in some wrong places. However, they can still span a Lie algebra over
$\RR$, and this is what we are now going to establish, where we first
need a better way to express $\cA \ts X_{\cB}$.

\subsection{Lie algebra structure}

Consider the formal sum of partitions defined by
\begin{equation}\label{eq:arrow}
      \cA\al\cB \, \defeq \sum_{A\in\cA}  \bigl( (\cA\setminus A) 
               \sqcup \cB |_A  - \cA \bigr) ,
\end{equation}
which clearly satisfies $\cA\al\cA = 0$. Also, for the formal sum
$\cC = \sum_{\pmax \ne \cA \in \cP (S)} n^{}_{\nts\cA} \ts \cA$, we set
\begin{equation}\label{eq:formal-sum}
    X^{}_{\cC} \, \defeq \sum_{\pmax \ne \cA \in \cP (S)} 
    n^{}_{\nts\cA} \ts X^{}_{\!\cA} \ts ,
\end{equation}
which in particular includes
$X^{}_{\! \cA\al \cB} = \sum_{A\in\cA} \bigl( X^{}_{(\cA\setminus
  A)\sqcup \ts \cB |_A} - X^{}_{\! \cA} \bigr)$ and analogously for
$X^{}_{\cB\al\cA}$.  So, we have $\cA \ts X^{}_{\cB} = \cA\al \cB$,
which allows us to work on the level of formal sums to determine the
commutator relations between the $X_{\cB}$ as follows.

\begin{prop}\label{prop:Lie}
  Let\/ $S = \{ 1, 2, \ldots , n \}$ be fixed. Then, the recombination
  Markov generators\/ $X^{}_{\!\cA}$ with\/
  $\pmax \ne \cA \in \cP (S)$ satisfy the commutation relations
\[   
     [X^{}_{\cB} , X^{}_{\cC} ] \, = \, X^{}_{\cB} - X^{}_{\cC}  
           + X^{}_{\cB \al \cC} - X^{}_{\cC \al\cB}
\]     
with the interpretation of the last two terms according to
\eqref{eq:arrow} and \eqref{eq:formal-sum}.
\end{prop}

\begin{proof}
  The identity holds if, for every basis vector $\cA\in\cP(S)$, both
  sides act equally on it (to the left). Since
  $\cA \ts X^{}_{\cB} = \cA\al\cB$, we thus have to show that
\begin{equation}\label{eq:to-show}
    (\cA\al\cB)\al\cC - (\cA\al\cC)\al\cB \, = \, \cA\al
    (\cB + \cB \al \cC - \cC - \cC \al \cB)
\end{equation}
holds for all $\cA\in\cP(S)$ and all
$\cB , \cC \in \cP(S)\setminus \pmax$, where it is important to note
that the action of $\al$ it \emph{not} associative. The left-hand side
(LHS) of \eqref{eq:to-show} evaluates as
\begin{align*}
\allowdisplaybreaks
  \text{LHS} \, & = \sum_{A\in\cA} 
       \bigl[ (\cA \sm A)\sqcup \cB |_A - \cA \ts \bigr]\al \cC
       - \bigl[ (\cA \sm A)\sqcup \cC |_A -\cA \ts \bigr] \al \cB \\[1mm]
      & =   \sum_{A\in\cA}  \bigl( \cA \al \cB - \cA\al\cC \bigr) \\
      & \quad \; +
      \sum_{A\in\cA} \Bigl( \sum_{\substack{A' \in\cA \\ A' \ne A}}
         \bigl[ (\cA\sm A \sm A')\sqcup \cB |_A \sqcup \cC |_{A'} 
                  - (\cA \sm A) \sqcup \cB |_A \bigr]  \\
      & \qquad\qquad + \sum_{B\in \cB |_A} \bigl[ (\cA \sm A) 
                  \sqcup (\cB |_A \sm B) \sqcup \cC |_B
                  - (\cA \sm A) \sqcup \cB |_A \bigr] \Bigr) \\
        & \quad \; -      
        \sum_{A\in\cA} \Bigl( \sum_{\substack{A' \in\cA \\ A' \ne A}}
          \bigl[ (\cA\sm A \sm A')\sqcup \cC |_A \sqcup \cB |_{A'} 
                  - (\cA \sm A) \sqcup \cC |_A \bigr]  \\
        & \qquad\qquad + \sum_{C\in \cC |_A} \bigl[ (\cA \sm A) 
                  \sqcup (\cC |_A \sm C) \sqcup \cB |_C
                  - (\cA \sm A) \sqcup \cC |_A \bigr] \Bigr)   .
\end{align*}     
Observing that the first contributions to both double sums contain
only terms that are symmetric in $A$ and $A'$ and thus cancel each
other, the LHS simplifies to
\begin{align*}
\allowdisplaybreaks
  \text{LHS} \, & = \, \lvert \cA \rvert
         \bigl( \cA \al \cB - \cA \al \cC \bigr) \\
  & \quad \; + \sum_{A\in\cA} \Bigl( \bigl( \lvert\cA\rvert - 1 \bigr)
      \bigl[ (\cA\sm A) \sqcup \cC |_A - (\cA\sm A) \sqcup \cB |_A \bigr] \\
  & \qquad \qquad + \sum_{B\in \cB |_A}   \bigl[ (\cA\sm A) \sqcup 
      (\cB |_A \sm B) \sqcup \cC |_B  - (\cA\sm A) \sqcup \cB |_A \bigr] \\
  & \qquad \qquad - \sum_{C\in \cC |_A}   \bigl[ (\cA\sm A) \sqcup 
    (\cC |_A \sm C) \sqcup \cB |_C  - (\cA\sm A) \sqcup \cC |_A \bigr]
       \Bigr) \\
  & = \, \cA \al \cB - \cA \al \cC \\
  & \quad \; + \sum_{A\in\cA} \sum_{B' \in \cB |_A} \bigl[ (\cA\sm A) \sqcup 
    (\cB |_A \sm B' ) \sqcup \cC |_{B'}  - (\cA\sm A) \sqcup \cB |_A \bigr] \\
  & \quad \; - \sum_{A\in\cA} \sum_{C' \in \cC |_A}  \bigl[ (\cA\sm A) \sqcup 
    (\cC |_A \sm C') \sqcup \cB |_{C'}  - (\cA\sm A) \sqcup \cC |_A \bigr]  ,   
\end{align*}
because the first two lines simplify as shown, and produce two terms
that also appear on the RHS. We thus only need to look at
\begin{align*}
\allowdisplaybreaks
  \cA \al \bigl( \cB \al \cC - \cC \al \cB \bigr) \, & = \,  \cA \al
    \Bigl( \sum_{B\in \cB} \bigl[ (\cB \sm B) \sqcup \cC |_B - \cB \bigr]
    - \sum_{C\in\cC} \bigl[ (\cC \sm C) \sqcup \cB |_C - \cC\bigr] \Bigr)
        \\[1mm]
    & = \sum_{A\in\cA} \sum_{B\in\cB} \bigl[ (\cA\sm A) \sqcup 
      \bigl( (\cB \sm B) \sqcup \cC |_B \bigr) |_A - (\cA\sm A)
         \sqcup \cB |_A\bigr]  \\[1mm]
    & \quad \; - \sum_{A\in\cA} \sum_{C\in \cC} \bigl[ (\cA\sm A) \sqcup 
      \bigl( (\cC \sm C) \sqcup \cB |_C \bigr) |_A - (\cA\sm A)
          \sqcup \cC |_A\bigr] .
\end{align*}
So, our claim follows when
\[
\begin{split}
  \sum_{B\in\cB} & \bigl[ (\cA\sm A) \sqcup \bigl( (\cB \sm B) \sqcup
     \cC |_B \bigr) |_A - (\cA\sm A) \sqcup \cB |_A \bigr] \\
  & = \sum_{B' \in \cB |_A} \bigl[ (\cA\sm A) \sqcup (\cB |_A \sm B') 
     \sqcup \cC |_{B'} - (\cA \sm A) \sqcup \cB |_A \bigr]
\end{split}  
\]
holds for all $A\in\cA$, and analogously for the second sum. But these
identities are the ones from Lemma~\ref{lem:technical}, and the
commutator identities hold as claimed.
\end{proof}

Observe that the generators $X^{}_{\! \cA}$ with
$\pmax \ne \cA \in \cP(S)$ span a real vector space, which is a
subspace of $\Mat(\dm,\RR)$ of dimension $\dm - 1$. The 
generators are linearly independent due to the
structure of the first row of the generators.  Now,
Proposition~\ref{prop:Lie} in conjunction with the bilinear extension
\eqref{eq:lin-extend} implies the following important result.

\begin{theorem}\label{thm:Lie}
  Let\/ $S = \{ 1, 2, \ldots , n \}$ with\/ $n\in\NN$ be fixed. Then,
  the recombination Markov generators\/ $X^{}_{\!\cA}$ with\/
  $\pmax \ne \cA \in \cP (S)$ span a real matrix Lie algebra of
  dimension\/ $\bell_n - 1$, where\/ $\bell_n$ is the\/ $n$-th Bell
  number.  \qed
\end{theorem}

\begin{remark}
  Let us emphasise that we started from a well-established model of
  population genetics, and were led to consider commutators, because
  the recombination Markov generators do \emph{not} form an algebra
  under matrix multiplication (for $n\geqslant 4$). While
  Lie-algebraic structures in models of genetics have been considered,
  see \cite{JS} and references therein, many models also form matrix
  algebras, which is certainly the case in phylogenetics. Still,
  Lie-algebraic techniques have been used, but were then less
  essential. Here, we have one of the first models where the Lie
  algebra structure is essential due to the absence of any matrix
  algebra.  \exend
\end{remark}

The marginalisation structure is reflected in a hierarchy of
subspaces, on which the Markov matrices and generators act via tensor
products. We next describe this for the generators, and get the
corresponding structure for the Markov matrices via a matrix
exponential.

\subsection{Tensor product structure}

If $\cA = \{ A_1, \ldots , A_k \}$, we interpret the corresponding
unit vector as a tensor product, namely
$\cA = \{ A_1 \} \otimes \{ A_2 \} \otimes \cdots \otimes \{ A_k \}$
in our partition-indexed vector notation, where
$\{ A_i \} = \pmax^{\nts A_i} \in \cP (A_i)$.  Here, we apply the
implicit (lexicographic) ordering along the parts in increasing
length.  Then, we can write
\[
    \cA  \ts X^{}_{\cB} \, = \sum_{i=1}^{k} \{ A_1\} \otimes \cdots \otimes
    \{ A_{i-1} \} \otimes \{ A_i\} X^{}_{\cB |_{A_i} } \! \otimes \{ A_{i+1} \}
    \otimes \cdots \otimes \{A_k \}
\]
where
$\{ A_i \} \ts X^{}_{\cB |_{A_i}} \! = \ts \cB |_{A_i} - \{ A_i\}$,
and thus
\[
   \cA \ts X^{}_{\cB} \, = \, \cA \sum_{i=1}^{k}
   \one \otimes \cdots \otimes \one \otimes X^{}_{\cB |_{A_i}} \!
   \otimes \one \otimes \cdots \otimes \one \ts .
\]
In fact, for any $\cC \preccurlyeq \cA$, we have
$\cC = \cC |_{A_1} \otimes \cdots \otimes \cC |_{A_k}$ and
$X^{}_{\cB}$ has the corresponding action on this product. Thus, if we
consider the subspace
$V^{\nts \cA} \defeq \langle \cC : \cC \preccurlyeq \cA
\rangle^{}_{\RR}$ and set
\begin{equation}\label{eq:tensor-Q}
    X^{\nts\cA}_{\cB} \, \defeq \sum_{i=1}^{k}
    \one \otimes \cdots \otimes \one \otimes X^{}_{\cB |_{A_i}} \!
    \otimes \one \otimes \cdots \otimes \one \ts ,
\end{equation}
we see that $X^{\nts \cA}_{\cB }$ acts on $V^{\nts \cA}$. In
particular, for any $\cC\preccurlyeq\cA$, one has
$\cC \ts X^{}_{\cB} = \cC \ts X^{\nts \cA}_{\cB}$. The definition from
\eqref{eq:tensor-Q} behaves well under addition and scalar
multiplication, via
$X^{\nts \cA}_{\alpha \cB + \beta \cC} = \alpha X^{\nts \cA}_{\cB} +
\beta X^{\nts \cA}_{\ts\cC}$. In fact, for any
$\varnothing\ne A\in\cA$, one also has
\[
    [ X^{}_{\cB |_{A}} , X^{}_{\cC |_{A}} ] \, = \, 
       [ X^{}_{\cB} , X^{}_{\cC} ] \big|_{A} ,
\]
because this boils down to the validity of
$(\cB \al \cC)|_A = \cB |_A \al \cC |_A$, which is nothing but the
identity from Corollary~\ref{coro:technical}. Applying this to all
$A = A_i$ with $1\leqslant i \leqslant k$, and using the commutativity
of the summands in \eqref{eq:tensor-Q}, one obtains the relation
\[
    [ X^{\nts \cA}_{\cB} , X^{\nts \cA}_{\cC} ] \, = \, 
      [ X^{}_{\cB} , X^{}_{\cC} ]^{\cA} .
\]
We have thus derived the following result.

\begin{prop}
  For any fixed\/ $\cA \in \cP (S)$, the mapping\/
  $X^{}_{\cB} \mapsto X^{\nts \cA}_{\cB}$ together with its\/
  $\RR$-linear extension defines a Lie algebra homomorphism.  \qed
\end{prop}

\begin{remark}\label{rem:single-exp}
  Let us note that the tensor product structure in \eqref{eq:tensor-Q}
  is preserved under the exponential map, where the standard relation
\[
   \exp ( \one \otimes \dots \otimes \one \otimes 
   X^{}_{\cB |_{A_i}} \! \otimes \one  \otimes \dots \otimes \one ) 
   \, = \, \one \otimes \dots \otimes \one \otimes
  \exp ( X_{\cB |_{A_i}}) \otimes \one \otimes \dots \otimes \one
\]
  together with the mutual commutativity of the summands in 
  \eqref{eq:tensor-Q} implies the identity
\[
    \exp ( X^{\nts \cA}_{\cB} ) \, = \,
        \bigotimes_{i=1}^{k} \exp \bigl( X^{}_{\cB |_{A_i}}\bigr) ,
\]
and the analogous one for any linear combinations of the
$X_{\cB}$. Indeed, the exponential of the sum thus is a product of
exponentials, each of which differs from $\one$ only at position $i$,
so that the factors give a single tensor product as shown. With
hindsight, this shows how the marginalisation structure
\eqref{eq:Q-marg} of generators gives rise to condition
\eqref{eq:M-marg} for Markov matrices, thus reversing the tangent
space argument from Section~\ref{sec:prelim}.  \exend
\end{remark}

Looking back at Section~\ref{sec:reco-part}, one can see the product
structure as follows. Let $Q$ be an RRS Markov generator for $n$
sites, and fix some $\cA = \{ A_1, A_2, \ldots , A_m \} \in \cP
(S)$. Now, let $Q\big|_{U}$ denote the restriction of $Q$ to
$U \subseteq S$, as defined by the marginal rates $\varrho^{U}_{\cD}$
from \eqref{eq:marg-rates} with $\cD \in \cP (U)$, and consider the
induced generator
$Q^{\cA} \defeq \bigl( Q_{\cB \cC}\bigr)_{\cB, \cC \preccurlyeq \cA}$.
Then, since $Q^{\cA}\big|_{A_i} = Q\big|_{A_i}$ holds for all
$1\leqslant i \leqslant m$, the statement of Eq.~\eqref{eq:Q-marg} can
be reformulated as
\[
  Q^{\cA} \, = \ts \sum_{i=1}^{m} \one \otimes \dots \otimes
  \one \otimes Q\big|_{A_i} \! \otimes \one \otimes \dots
  \otimes \one \ts ,
\]
which also implies
$\exp \bigl(Q^{\cA} \bigr) = \bigotimes_{i=1}^{m} \exp \bigl(
Q\big|_{A_i} \bigr)$. This is fully consistent with the product
structure of the RMS Markov matrices. Indeed, if $M\big|_{U}$ is the
restiction of $M$ to $U$ as defined via the marginal probabilities
$r^{U}_{\cD}$ from \eqref{eq:marg-prob} with $\cD\in\cP(U)$,
Eq.~\eqref{eq:M-marg} means nothing but
$M^{\cA} = \bigotimes_{i=1}^{m} M\big|_{A_i}$, again due to
$M^{\cA}\big|_{A_i} = M\big|_{A_i}$ for all $i$, so also
$\exp(Q)^{\cA} \big|_{A_i} = \exp (Q) \big|_{A_i}$.

\subsection{Embedding conditions}\label{sec:gen-emb-cond}

To approach the embedding problem, we will first make use of the
triangular form of the matrices and employ the results from
\cite{tle}, adapted to our setting with the lattice of partitions,
which simplifies several of the sums occurring. For the convenience of
the reader, we will recall all relevant formulas, but refer to
\cite{tle} for the proofs.

So, let $M$ be a recombination Markov matrix with marginalisation
structure \eqref{eq:M-marg}, and assume that $M$ has simple
spectrum. Then, it is diagonalisable, and can be written as
\[
    M \, = \, T  D \, T^{-1} ,
\]
where $D$ is a diagonal matrix that agrees with the diagonal of $M$.
Note that $T = (\tau^{}_{\!\cA \cB})^{}_{\cA, \cB \in \cP(S)}$
columnwise contains the right eigenvectors of $M$, while the rows of
$T^{-1}$ are the left eigenvectors. Both $T$ and $T^{-1}$ are upper
triangular. We can assume them to be real because all eigenvalues
$\lambda^{}_{\cC} = M^{}_{\cC \cC}$ of $M$ are real. Now, we have
\[
  M^{}_{\! \cA \ts \cB} \, = \! \sum_{\cB \preccurlyeq \udo{\cC} \preccurlyeq \cA}
  \! \tau^{}_{\!\cA \ts \cC} \ts \tilde{\tau}^{}_{\ts \cC \cB} \, \lambda^{}_{\cC}
\]
for $\cB \preccurlyeq \cA$, and $M^{}_{\! \cA\ts \cB}=0$ otherwise.
Here, $\tau^{}_{\!\cA \ts \cB}$ and $\tilde{\tau}^{}_{\ts \cC\cD}$ denote
the elements of $T$ and $T^{-1}$, respectively. Since each right
eigenvector is unique up to an overall non-zero factor, which is then
compensated for in the left eigenvectors by the reciprocal of this
factor, we set
\begin{equation}\label{eq:def-theta}
  \vartheta^{}_{\!\cA} (\cC, \cB) \, \defeq \,
   \tau^{}_{\!\cA \ts \cC} \ts \tilde{\tau}^{}_{\ts \cC \cB} \ts ,
\end{equation}
which is blind to this freedom in the choice of $T\nts$. We then have
$M^{}_{\! \cA\ts \cB} = \sum_{\cB \preccurlyeq \udo{\cC} \preccurlyeq
  \cA} \vartheta^{}_{\!\cA} (\cC, \cB) \, \lambda^{}_{\cC}$ together
with
\begin{equation}\label{eq:magic-0}
  \sum_{\cA \succcurlyeq \udo{\cC} \succcurlyeq\cB}
  \! \vartheta^{}_{\!\cA} (\cC, \cB) \, = \, \delta^{}_{\! \cA \ts \cB}
  \quad \text{ and} \quad
  \sum_{\cC \succcurlyeq \udo{\cE} \succcurlyeq\cD} \!
  \vartheta^{}_{\! \cA} (\cC, \cE) \, \vartheta^{}_{\cE} (\cD,\cB)
  \, = \, \delta^{}_{\cC \cD} \, \vartheta^{}_{\! \cA} (\cC, \cB)
\end{equation}
for all $\cA, \cB \in \cP (S)$ with $\cA \succcurlyeq \cB$ in the
first relation and all $\cA, \cB, \cC, \cD \in \cP (S)$ subject to
$\cA \succcurlyeq \cC \succcurlyeq \cD \succcurlyeq \cB$ in the
second. Note that the first is a consequence of $T \ts\ts T^{-1} = \one$,
while the second follows easily from $T^{-1} T = \one$. Note also that
$\vartheta^{}_{\!\cA} (\cC, \cB) = 0$ whenever
$\cA \not \succcurlyeq \cC$ or $\cC \not \succcurlyeq \cB$, due to the
upper-triangular nature of $T$ and ${T^{-1}\!}$. The latter also
implies that we have $\vartheta^{}_{\!\cA} (\cA, \cA) = 1$ for all
$\cA \in \cP (S)$.

The $\vartheta$-coefficients satisfy another identity, namely
\begin{equation}\label{eq:magic-1}
  \sum_{\cA \succcurlyeq \udo{\cD} \succcurlyeq \cC} \!
  M^{}_{\!\cA\ts\cD} \, \vartheta^{}_{\cD} (\cC, \cB)
  \, = \, \lambda^{}_{\cC} \, \vartheta^{}_{\!\cA} (\cC , \cB) 
\end{equation}
for all $\cA, \cB, \cC \in \cP (S)$ with
$\cA \succcurlyeq \cC \succcurlyeq \cB$, which follows from
\cite[Lemma~2.1]{tle}.  Since $M$ has simple spectrum by assumption,
Eq.~\eqref{eq:magic-1} can be solved for
$\vartheta^{}_{\! \cA} (\cC , \cB)$ with $\cA \succ \cC$ to obtain
\begin{equation}\label{eq:magic-2}
  \vartheta^{}_{\!\cA} (\cC, \cB) \, = \,
  \myfrac{1}{\lambda^{}_{\cC} - \lambda^{}_{\!\cA}}
  \sum_{\cA \succ \udo{\cD} \succcurlyeq \cC} \!
   M^{}_{\!\cA\ts\cD} \, \vartheta^{}_{\cD} (\cC, \cB) \ts ,
\end{equation}
which leads to the following result.

\begin{lemma}\label{lem:theta}
  Let\/ $M$ be a recombination Markov matrix with simple spectrum.
  Then, the\/ $\vartheta$-coefficients from \eqref{eq:def-theta} are
  unique, and the non-zero ones can be computed recursively via
\[
  \vartheta^{}_{\!\cA} (\cC, \cB) \, = \, \delta^{}_{\! \cA \ts \cC}
  \Bigl( \delta^{}_{\! \cA \ts \cB} \, - \!
      \sum_{\cA \succ \udo{\cE} \succcurlyeq \cB}
         \! \vartheta^{}_{\!\cA} (\cE, \cB) \Bigr)
         + \myfrac{1}{\lambda^{}_{\cC} - \lambda^{}_{\!\cA}}
      \sum_{\cA \succ \udo{\cD} \succcurlyeq \cC} \! M^{}_{\!\cA\ts\cD}
         \, \vartheta^{}_{\cD} (\cC , \cB) \ts ,
\]
with empty sums understood to be $0$.  This encodes a complete
recursion of the\/ $\vartheta^{}_{\!\cA} (\cC, \cB)$ for\/
$\cA \succcurlyeq \cC \succcurlyeq \cB \succcurlyeq \pmin\ts$, with
the initial conditions\/ $\vartheta^{}_{\!\cA} (\cA, \cA) = 1$ for\/
$\cA \in \cP (S)$.
\end{lemma}

\begin{proof}
  The claimed uniqueness follows directly from the definition in
  \eqref{eq:def-theta} as indicated earlier, because $M$ having simple
  spectrum means that any (non-zero) factor to a column of $T$, which
  is the only remaining freedom in the choice of $T$, is compensated
  by the reciprocal prefactor to the corresponding row of $T^{-1}$,
  hence cancels in the $\vartheta$-coefficients.

  Simple spectrum also means that we can use \eqref{eq:magic-2} for
  $\cA\succ\cC$. Now, we augment this with a relation for
  $\vartheta^{}_{\!\cA} (\cA, \cB)$, which we can get from the first
  identity in \eqref{eq:magic-0}. Adding the initial conditions as
  stated earlier, we obtain the claimed identity, and it is not hard to
  check that this constitutes a complete recursion for the non-trivial
  parts of the $\vartheta$-coefficients.
\end{proof}

The crucial point now is the following. When $M$ is non-singular and
has simple spectrum, it possesses a unique real matrix logarithm, $Q$
say, which also has simple spectrum and commutes with $M$, hence
possesses the \emph{same} eigenvectors as $M$. So, we can use the
$\vartheta$-coefficients of $M$ to calculate the first row of $Q$ as
\begin{equation}\label{eq:Q-condi}
  Q^{}_{\pmax \cA} \, = \sum_{\udo{\cB} \ts \succcurlyeq \cA}
  \mu^{}_{\cB} \, \vartheta^{}_{\pmax} (\cB, \cA) \, =
  \sum_{\udo{\cB} \ts \succcurlyeq \cA}
   \log (\lambda^{}_{\cB}) \, \vartheta^{}_{\pmax} (\cB, \cA) \ts .
\end{equation}
Consequently, we can express the non-negativity conditions
$Q^{}_{\pmax \cA} \geqslant 0$ via the eigenvalues of $M$, which are
positive and distinct by assumption, and the recursively determined
$\vartheta$-coefficients.

Specialising these conditions to the cases of $2$, $3$ and $4$ sites
gives the conditions derived above by direct means, where one has to
observe that we assume $M$ to have positive, simple spectrum. For $2$
and $3$ sites, the $\vartheta$-coefficients are parameter independent,
in line with the commutativity of the matrix algebra in these two
cases. However, this approach does not show which of the conditions
are satisfied automatically. In particular, observing that the
$\vartheta$-coefficients are real but signed, the interpretation of
the alternating sums seems difficult. At this point, the general
situation can be stated as follows.

\begin{theorem}
  Let\/ $M$ be a non-singular RMS Markov matrix for\/ $n$ sites, and
  set\/ $A=M\nts - \one$. Then, the following properties are
  equivalent.
\begin{enumerate}\itemsep=2pt
\item The matrix\/ $M$ is embeddable.
\item $M$ is embeddable with a Markov generator of RRS type.
\item The principal matrix logarithm $\log (\one + A)$ is
  a Markov generator.  
\end{enumerate}
The embedding is unique when no elementary Jordan block in the JNF 
of\/ $M$ occurs more than once. In particular, if\/ $M$ has simple spectrum,
$(3)$ is equivalent with the\/ $Q^{}_{\pmax \cA}$ from \eqref{eq:Q-condi} 
being non-negative for all\/ $\pmax \ne \cA \in \cP(S)$.

Further, when an RMS Markov matrix\/ $M$ with multiple elementary Jordan 
blocks is embeddable, the embedding is not unique, but the principal logarithm
of\/ $M$ is the only Markov generator of RRS type, and no other real
logarithm is upper triangular.
\end{theorem}

\begin{proof}
  Any non-singular RMS Markov matrix satisfies
  $\sigma (M) \subset \RR_{+}$.  When the spectrum is simple or, more
  generally, when no elementary Jordan block occurs more than once,
  there is only one real matrix logarithm of $M$ by Culver's result
  (Fact~\ref{fact:Culver}). The latter is the principal matrix
  logarithm, which must be a matrix of RRS type, because there is
  always one real matrix logarithm of $M$ of this type by
  Theorem~\ref{thm:RMS-RRS}.  In this case, the equivalence of
  conditions (1), (2) and (3) is clear, and the embedding must be
  unique.
   
  When $M$ is diagonalisable, $\sigma(M)$ is free of multiplicities if
  and only if $\sigma (M)$ is simple. Then, condition (3) is indeed
  equivalent with the non-negativity of the $Q^{}_{\pmax \cA}$ from
  \eqref{eq:Q-condi}, for all $\pmax \ne \cA \in \cP (S)$, by an
  application of Lemma~\ref{lem:theta}.
   
  For the cases with degenerate spectrum, we can once again
  approximate $M$ with RMS Markov matrices with simple spectrum. Since
  the embeddable Markov matrices form a set that is relatively closed
  within the set of all Markov matrices with positive determinant, we
  find uniquely embeddable ones in any small neighbourhood of $M$, and
  use a standard limit theorem to get embeddability of $M$ also via
  its principal matrix logarithm. When further embeddings exist, we
  know from Theorem~\ref{thm:RMS-RRS} that no other one is upper
  triangular, hence also not of RRS type, and our argument is
  complete.
\end{proof}

To develop the picture further, it seems likely that one has to
investigate the algebraic structure in more detail. When $r\in\RR^d$
is a vector with row sum $1$ with positive spectrum (in the sense of
Definition~\ref{def:spec}), let $r \mapsto M(r)$ be the mapping to the
corresponding RMS matrix, which is injective. Likewise, let
$\rho \mapsto Q(\rho)$ denote the mapping from vectors with $0$ row
sum to an RRS matrix, which is injective as well. If
$\varphi : \Mat (d, \RR) \longrightarrow \RR^d$ is the projection to
the row vector that is the first row of the input matrix, we can
interpret $\rho \mapsto \varphi \bigl(\exp(Q(\rho))\bigr)$ as the
exponential map from the Lie algebra to the component of the Lie group
connected to the identity, while
$r \mapsto \varphi\bigl(\log (M(r))\bigr)$ is the matching logarithm.
Since this is a bijection, one further step could be to derive an
explicit version of these maps and study them, aiming at a better
interpretation of the embedding condition, which says that a
probability vector $p$ belongs to an embeddable RMS Markov matrix if
and only if $\rho = \log (p)$ has all entries except possibly the
first one non-negative.

In fact, by a straightforward extension of \cite{BvW14} and
\cite{BE18}, the mapping
$\rho \mapsto \varphi \bigl(\exp(Q(\rho))\bigr)$ may be constructed
explicitly via a partial tree decomposition combined with
combinatorial tools such as inclusion/exclusion and M\"{o}bius
inversion, and likewise for the mapping
$r \mapsto \varphi\bigl(\log (M(r))\bigr)$. However, due to the
alternating signs resulting from inclusion/exclusion, it is not clear
how to arrive at explicit non-negativity criteria. We thus leave this
as a challenge for future work.

\section*{Acknowledgements}

MB is grateful to the University of Tasmania in Hobart for
hospitality, where part of this work was done. We acknowledge support
by the German Research Foundation (DFG), within the CRC 1283/2
\mbox{(2021 - 317210226)} at Bielefeld University.

\end{document}